\documentclass[reqno, 11pt]{amsart} 

\usepackage{graphicx,cite}
\usepackage{mathrsfs,amsfonts,amssymb,amsmath}
\newcommand{\norm}[1]{\left\Vert#1\right\Vert}
\newcommand{\abs}[1]{\left\vert#1\right\vert}

\newcommand{\R}{\mathbb R}
\newcommand{\D}{\partial}
\newcommand{\eps}{\varepsilon}

\newcommand{\dv}{\mathrm{div}\,}

\newcommand{\dm}{{\mathscr{D}_t}}
\newcommand{\idm}{{\mathscr{D}}}
\newcommand{\nb}{\nabla}
\newcommand{\bnb}{\overline{\nb}}
\newcommand{\tr}{\mathrm{tr}\,}
\newcommand{\curl}{\mathrm{curl}\,}
\newcommand{\dist}{\mathrm{dist}\,}
\newcommand{\ls}{\leqslant}
\newcommand{\gs}{\geqslant}
\newcommand{\N}{\mathcal{N}}

\newcommand{\vol}{\mathrm{Vol}\,}
\newcommand{\sgn}{\mathrm{sgn}}

\newcommand{\no}{\nonumber}
\newcommand{\K}{\mathcal{K}}
\newcommand{\E}{\mathcal{E}}
\newcommand{\T}{\mathscr{T}}
\newcommand{\F}{\mathbb{F}}
\newcommand{\detail}[1]{#1}%

\usepackage{ifpdf}
\ifpdf
\usepackage[CJKbookmarks=true,
         hyperindex=true,
         pdfstartview=FitH,
         bookmarksnumbered=true,
         bookmarksopen=true,
         colorlinks=true,
         citecolor=blue,
         linkcolor=blue,
         urlcolor=blue,
         pdfborder=001,
         pdfauthor={Chengchun Hao},
         pdftitle={},
         pdfkeywords={},
         ]{hyperref}
\else
\usepackage[
         hypertex,
         hyperindex,
         linkcolor=blue,
         unicode,
         citecolor=blue%
         ]{hyperref}
\fi

\allowdisplaybreaks
\numberwithin{equation}{section}
\usepackage{amsthm}
\newtheorem{theorem}{Theorem}[section]

\newtheorem{lemma}[theorem]{Lemma}

\theoremstyle{definition}
\newtheorem{definition}[theorem]{Definition}
\theoremstyle{remark}

\begin{document}


\title[Free boundary problem of incompressible elastodynamics]
{A Priori Estimates for the Free Boundary Problem of Incompressible Neo-Hookean Elastodynamics}%

\author{Chengchun Hao}
\address{Institute of Mathematics,
Academy of Mathematics and Systems Science,
and Hua Loo-Keng Key Laboratory of Mathematics,
Chinese Academy of Sciences,
Beijing 100190,  and BCMIIS, Beijing 100048,  China}
\email{hcc@amss.ac.cn}

\author{Dehua Wang}
\address{Department of Mathematics, University of Pittsburgh,
                           Pittsburgh, PA 15260, USA.}
\email{dwang@math.pitt.edu}



\begin{abstract}
A free boundary problem for the incompressible neo-Hookean elastodynamics   is studied in two and three spatial dimensions. The \textit{a priori} estimates  in Sobolev norms of solutions  with the physical vacuum condition are established through a geometrical point of view of \cite{CL00}.  
Some estimates on  the second fundamental form and  velocity of the free surface are also obtained.
\end{abstract}


\date{December 31, 2013}
\maketitle


\section{Introduction}

We are concerned with the motion of neo-Hookean elastic waves in an incompressible material 
for which the deformation or strain   is proportional to the  stress.
Precisely, we consider the free boundary problem of the following incompressible elastodynamic equations  of neo-Hookean elastic materials:
\begin{subequations}\label{eld1}
  \begin{align}
    & v_t+ v\cdot\D  v+\D p=\dv(FF^\top),   \label{eld1.1}\\
    & F_t+ v\cdot \D F=\D v F , \label{eld1.2}\\
    &\dv  v=0,  \quad \dv F^\top=0,\label{eld1.3}\end{align}
\end{subequations}
in a set $$\idm=\underset{0\ls t\ls T}{\bigcup}\{t\}\times \idm_t,$$ 
where $\idm_t\subset \R^n, n=2 \text{ or } 3$, is the domain that the material  occupies at time $t\in [0,T]$
  for some $T>0$ ; where $\D=(\D_1,\cdots,\D_n)$ and $\dv$ are the usual gradient operator and spatial divergence in the Eulerian coordinates with   $\D_i=\D/\D{x^i}$, respectively;
 $v(t,x)=(v_1(t,x),\cdots, v_n(t,x))$ is the velocity vector field of the fluid, $p(t,x)$ is the pressure, 
 $F(t,x)=(F_{ij}(t,x))$ is the deformation tensor,  $F^\top=(F_{ji})$ denotes the transpose of the $n\times n$ matrix 
 $F$, $FF^\top$ is the Cauchy-Green tensor in the case of neo-Hookean elastic materials (cf. \cite{Lei12,ST07});
 and the notations 
$(\D v)_{ij}=\D_jv_i$, $(\D vF)_{ij}= (\D vF)^{ij}=(\D v)_{ik}F^{kj}=\D_kv^iF^{kj}$, $\dv v=\D_i v^i$, 
$(\dv F^\top)^i=\D_jF^{ji}$  follow  the Einstein summation convention: $v^i=\delta^{ij}v_j=v_i$ and $F^{ij}=\delta^{ik}\delta^{jl}F_{kl}=F_{ij}$. 
The boundary conditions on the free boundary:
 $$\D\idm=\underset{0\ls t\ls T}{\bigcup}\{t\}\times \D\idm_t$$
are prescribed as the following:
\begin{subequations}\label{eld1'}
  \begin{align}
   p=&0 \;\text{ on }  \D\idm, \label{eld1.5}\\
 \N\cdot F^\top=&0 \;\text{ on } \D\idm,\label{eld1.4}\\
   (\D_t+v\cdot \D)|_{\D\idm}&\in T(\D\idm), \label{eld1.5'}
  \end{align}
\end{subequations}
where $\N(t,x)$ is the exterior unit normal to the free surface $\D\dm$ and $T(\D\idm)$ is the tangential space to $\D\idm$. The boundary condition \eqref{eld1.5} implies that the pressure $p$ vanishes outside the domain,   \eqref{eld1.4} indicates that the normal component of $F^\top$ (i.e., $\N_kF^{kj}$) vanishes on the boundary, and   \eqref{eld1.5'} means that the free boundary moves with the velocity $v$ of the material particles, i.e., $v\cdot\N=\kappa$ on $ \D\idm_t$ with $\kappa$ the normal velocity of $ \D\idm_t$.

For a simply connected bounded domain $\idm_0\subset \R^n$  that is homeomorphic to the unit ball, and the initial data $(v_0(x), F_0(x))$ satisfying the constraint \eqref{eld1.3}:
$\dv  v_0=0,  \; \dv F_0^\top=0$, we shall establish  \textit{a priori} estimates for the set $\idm\subset [0, T]\times \R^n$ and the vector fields $v$ and $F$ solving \eqref{eld1}-\eqref{eld1'} with the initial conditions:
\begin{equation}\label{initialcondition}
\{x: (0, x)\in \idm\}=\idm_0, \quad  ( v, F)|_{t=0}=( v_0(x), F_0(x))  \text{ for } x\in \idm_0.
\end{equation}
We will study the  free boundary problem \eqref{eld1}-\eqref{initialcondition} under the following natural condition (cf. \cite{bealhou,CL00,Coutand,Ebin,HLarma,L1,L2,LN,SZ,wu1,wu2,zhang}):
\begin{align}\label{eq.phycond}
  \nb_\N p\ls -\eps<0 \text{ on } \D\dm,
\end{align}
where $\nb_\N=\N^i\D_i$ and $\eps>0$ is a constant.
We  assume that  \eqref{eq.phycond} holds initially, and will verify that it still holds   within a time period.
Roughly speaking,   the elastic body will not break up in the interior since the pressure is positive, 
the boundary moves according to the velocity, and the boundary is the level set of the pressure that,  together with the Cauchy-Green tensor, determines the acceleration, thus the regularity of the boundary in quite involved, which is a difficult issue for this problem.

There have been some results for the free surface problem of the incompressible Euler equations of fluids in the recent decades, see for examples \cite{AM,CL00,Coutand,Ebin,L1,L2,LN,SZ,wu1,wu2,zhang} and the references therein. 
For elastodynamics, there have been some studies on the fixed boundary problems, see for examples 
 Ebin \cite{Ebin93,Ebin96} for the global existence of small solutions to the three-dimensional incompressible and isotropic elasticity equations  and the special case of incompressible neo-Hookean materials, and Sideris-Thomases  \cite{ST05,ST07} for the global existence of the three-dimensional incompressible elasticity.
 In this paper, we shall prove the \textit{a priori} estimates for the free boundary problem \eqref{eld1}-\eqref{initialcondition}  in all physical spatial dimensions $n=2,3$ by adopting a geometrical point of view used in Christodoulou-Lindblad \cite{CL00} and establishing estimates on quantities such as the second fundamental form and the velocity of the free surface.

Define the material derivative by $D_t=\D_t+v^k\D_k$. We rewrite the system \eqref{eld1} as
\begin{subequations}\label{eld11}
  \begin{align}
    & D_t  v_i+\D_i p=\D_kF_{ij}F^{kj},  && \text{ in } \idm, \label{eld11.1}\\
    & D_t F_{ij}=\delta^{kl}\D_k v_iF_{lj}, && \text{ in } \idm, \label{eld11.2}\\
    &\D_i  v^i=0,  \quad \D_jF^{ji}=0 , && \text{ in } \idm.\label{eld11.3}\end{align}
\end{subequations}
From \eqref{eld11}, one has
\begin{align} \label{ener}
\frac{1}{2}\frac{d}{dt}\int_{\idm_t}(\abs{v}^2+\abs{F}^2)dx=-\int_{\D\idm_t}pv^i\N_idS+\int_{\D\idm_t} F_{ij}F^{kj}v^i\N_k dS,
\end{align}
where $dS$ is the surface measure. We see that \eqref{ener} and the boundary conditions \eqref{eld1'} yield  the conserved physical energy: 
\begin{align}\label{energy.phys}
E_0(t)=\int_{\idm_t}\left(\frac{1}{2}\abs{v(t,x)}^2+\frac{1}{2}\abs{F(t,x)}^2\right)dx.
\end{align}
Note that the identities $\dv F^\top=0$ in $\idm$ and $\N\cdot F^\top=0$ on $\D\idm$ are preserved, that is, they hold if $\dv F_0^\top=0$ in $\idm_0$ and $N\cdot F_0^\top=0$ on $\D\idm_0$ for initial data,  where $N$ denotes the exterior unit normal to the initial interface $\D\idm_0$, which will be verified later in the Lagrangian coordinates.

The higher order energy norm has a boundary part and an interior part.  Following the definitions and notations of \cite{CL00},  we define the boundary part   through the orthogonal projection to the tangent space of the boundary.
The orthogonal projection $\Pi$ to the tangent space of the boundary of a $(0, r)$ tensor $\alpha$ is defined as the projection of each component in the normal direction, that is,
\begin{equation}\label{projection}
(\Pi \alpha)_{i_1\cdots i_r}=\Pi_{i_1}^{j_1}\cdots \Pi_{i_r}^{j_r} \alpha_{j_1\cdots j_r},  
\end{equation}
where $\Pi_i^j=\delta_i^j-{\N}_i{\N}^j$ with the convention ${\N}^j=\delta^{ij} {\N}_i={\N}_j$.

Denote  by $\bar \partial_i=\Pi_i^j \partial_j$   a tangential derivative. Then we see that for  $q=0$ on $\partial \idm_t$, one has $\bar \partial_i q=0$ on $\partial \idm_t$ and
\begin{equation}\label{esff1}
(\Pi \partial^2 q)_{ij}=\theta_{ij} \nabla_{\N} q,
\end{equation}
where ${\theta}_{ij}=\bar \partial_i \N_j$ is the second fundamental form of $\partial \idm_t$.
 
Consider the following  positive definite quadratic form $Q$ of the form (see \cite{CL00}):
\begin{equation}\label{quadratic form}
Q(\alpha, \beta)=q^{i_1j_1}\cdots q^{i_rj_r}\alpha_{i_1\cdots i_r}\beta_{j_1\cdots j_r},
\end{equation}
where
\begin{equation}
q^{ij}=\delta^{ij}-\eta(d)^2\N^i\N^j, \quad d(x)=\dist(x, \D\idm_t), \quad\text{and } \N^i=-\delta^{ij}\D_j d, \end{equation}
and     $\eta$ is a smooth cutoff function satisfying 
$$0\ls \eta(d)\ls 1, \quad 
\eta(d)= 1 \text{ for } d < \frac{d_0}4,\quad
\eta(d)=0 \text{ for }  d > \frac{d_0}2,$$
 with $d_0$   a fixed number  less than
the injectivity radius of the normal exponential map $\iota_0$ which is the largest
number $\iota_0$ such that the map
\begin{equation}
 \D\idm_t\times (-\iota_0, \iota_0)\to \{x\in \R^n:  \dist(x,\  \D\idm_t) < \iota_0\},
\end{equation}
defined by
$$(\bar x, \iota)\to x=\bar x+\iota \N(\bar x),$$
is an injection.
The quadratic form $Q$  is the inner product of the tangential components when restricted to the boundary:
 $Q(\alpha, \beta)=\langle\Pi \alpha, \Pi \beta\rangle$,  and  $Q(\alpha, \alpha)=|\alpha|^2$ in the interior.

Let $\sgn(s)$ be the sign  function of the real number $s$. Denote 
$$(\curl F^\top)_{ijk}:=\D_iF_{jk}-\D_jF_{ik}$$ and $\vartheta=(-\nabla_{\N} p)^{-1}$. Then, we define the higher order energies for $r\gs 1$ as:
\begin{align}\label{ereen}
E_r(t)=&\int_{\idm_t}\left(\delta^{ij}Q(\D^r v_i, \D^r v_j)+\delta^{ij}\delta^{km}Q(\D^r F_{ik}, \D^r F_{jm})\right)dx\no\\
       &+\int_{\idm_t}\left(|\D^{r-1}\curl v|^2+|\D^{r-1}\curl F^\top|^2\right)dx\no\\
       &+\sgn(r-1)\int_{\D\idm_t}Q(\D^r p, \D^r p)\vartheta dS.
\end{align}
The higher order energy norm has a boundary part (for $r\gs 2$)  which controls the norms of the second fundamental form of the free surface, and   an interior part which  controls the norms of the velocity  and thus the pressure. We will prove that the time derivatives of the energy norms are controlled by themselves. 
One advantage of  the above higher order energy norms is that the time derivatives of the interior parts yield    some boundary terms which have some cancellation with
the leading-order  terms in the   time derivatives of the boundary integrals.

Now, we can state the main result of this paper as follows.
\begin{theorem} \label{MAIN}
Let
$$\K(0)=\max\left(\norm{\theta(0,\cdot)}_{L^\infty(\D\idm_0)}, \frac{1}{\iota_0(0)}\right)$$ and 
  $$\E(0)=\norm{\frac{1}{\nb_N p(0,\cdot)}}_{L^\infty(\D\idm_0)}.
$$
Then, there exists a continuous function $\T>0$ such that if
  \begin{align*}
    T\ls \T(\K(0),\E(0),E_0(0),\cdots, E_{n+1}(0),\vol \idm_0 ),
  \end{align*}
  any smooth solution of the free boundary problem   \eqref{eld1}-\eqref{eq.phycond} satisfies
  \begin{align*}
    \sum_{s=0}^{n+1} E_s(t)\ls 2\sum_{s=0}^{n+1} E_s(0), \quad 0\ls t\ls T.
  \end{align*}
\end{theorem}

We remark that Theorem \ref{MAIN} extends the result of \cite{CL00} for the Euler equations of incompressible flow to the elastodynamics \eqref{eld1}.  Our proof will be based on the geometric point of view following \cite{CL00}. We   need to develop   new ingredients in the proof to handle the deformation $F$ and the interaction with the velocity $v$, which requires some new thoughts.
For the well-posedness of incompressible Euler equations we refer the readers to \cite{L1,L2} and the references therein.
The well-posedness of the elastodynamics \eqref{eld1} is much harder. 
In this paper we shall explore all the symmetries of the equations and then we will be able to establish the  sharp \textit{a priori} estimates. Although the well-posedness does not follow directly, these estimates are crucial for  
the local existence of smooth solution for the system \eqref{eld1}  which could be possibly obtained by improving the estimates of this paper together with the Nash-Moser technique. 

The rest of the paper is organized as follows. In Section \ref{sec.Lag}, we reformulate the problem  to a fixed initial-boundary value problem in the Lagrangian coordinates. The Lagrangian transformation induces a Riemannian metric on $\idm_0$, for which we recall the time evolution properties derived in \cite{CL00} and prove some new identities which will be used later. Section \ref{sec.1order} is devoted to the first order energy estimates. In Section \ref{sec.rorder}, we derive the higher order energy estimates using the identities derived in Section \ref{sec.Lag}.  We justify the \textit{a priori} assumptions in Section \ref{sec.justifi}.  For the sake of completeness and the convenience of the readers we add an appendix  to state some estimates which are used in this paper but were already proved in \cite{CL00}.

\bigskip

\section{Reformulation in Lagrangian Coordinates}\label{sec.Lag}

In this section, we shall introduce the Lagrangian   coordinates to reformulate the free boundary problem to fix the boundaries.

Following the same terminology and lines of \cite{CL00},  we present the transformation between the Eulerian coordinates $(t,x)$ and the Lagrangian coordinates $(t,y)$. 
For a velocity vector field $v(t,x)$   in $\idm\subset [0,T]\times\R^n$ with
$(1, v)\in T(\D\idm)$ (i.e., the boundary  $\D\dm$ moves with the velocity $v$),
define  $x=x(t,y)=f_t(y)$ as the trajectory of   particles:
\begin{align}\label{trajectory}
  \left\{\begin{aligned}
    &\frac{dx}{dt}= v(t,x(t,y)), \quad (t,y)\in [0,T]\times \Omega,\\
    &x(0,y)=f_0(y), \quad y\in\Omega.
  \end{aligned}\right.
\end{align}
where, 
$f_0:\Omega\to \idm_0$   is some given diffeomorphism defined on a  given simple  domain $\Omega$,
and  $f_t:\Omega\to \dm$ with $f_t(y)=x(t,y)$ is a change of coordinates for each $t$.
As a result, for each $t$ from the Euclidean metric $\delta_{ij}$ in $\dm$,  a metric $g_{ab}$ (with inverse  $g^{cd}$) in $\Omega$ is induced as
\begin{align}\label{metric1}
  g_{ab}(t,y)=\delta_{ij}\frac{\D x^i}{\D y^a}\frac{\D x^j}{\D y^b},
\end{align}
and thus the covariant differentiation $\nb_a$ in the $y_a$-coordinates, $a=0, \cdots, n$,  in $\Omega$ will be used for the metric $g_{ab}(t,y)$.
An $(0,r)$ tensor $w(t,x)$ in the $x$-coordinates can be expressed as $k(t,y)$  in the $y$-coordinates:
\begin{align*}
  k_{a_1\cdots a_r}(t,y)=\frac{\D x^{i_1}}{\D y^{a_1}}\cdots \frac{\D x^{i_r}}{\D y^{a_r}}w_{i_1\cdots i_r}(t,x), \quad x=x(t,y),
\end{align*}
and 
the covariant differentiation of the tensor $k(t,y)$ is the $(0,r+1)$ tensor:
\begin{align}\label{eq.covtensor}
  \nb_a k_{a_1\cdots a_r}=\frac{\D x^i}{\D y^a}\frac{\D x^{i_1}}{\D y^{a_1}}\cdots \frac{\D x^{i_r}}{\D y^{a_r}}\frac{\D w_{i_1\cdots i_r}}{\D x^i}
\end{align}
with the invariant norms of tensors: 
\begin{align}\label{eq.norminv}
  g^{a_1b_1}\cdots g^{a_rb_r} k_{a_1\cdots a_r}k_{b_1\cdots b_r}=\delta^{i_1j_1}\cdots \delta^{i_rj_r}w_{i_1\cdots i_r}w_{j_1\cdots j_r}.
\end{align}
From
\begin{align}\label{eq.Di}
  \D_i=\frac{\D}{\D x^i}=\frac{\D y^a}{\D x^i}\frac{\D}{\D y^a},
\end{align}
the curvature vanishes in both the $x$-coordinates and the $y$-coordinates,  and then  
$[\nb_a,\nb_b]=0$.
Denote
 $${{k_{a\cdots}}^b}_{\cdots c}=g^{bd}k_{a\cdots d\cdots c}$$
 and 
 \begin{align*}
  D_t=\left.\frac{\D}{\D t}\right|_{y=\textrm{const}}=\left.\frac{\D}{\D t}\right|_{x=\textrm{const}}+v^k\frac{\D}{\D x^k}.
\end{align*}
Since covariant differentiation commutes with lowering and rising indices: 
$$g^{ce}\nb_a k_{b\cdot e\cdots d}=\nb_a g^{ce}k_{b\cdot e\cdots d},$$ 
one has,  from \cite[Lemma 2.2]{CL00}, 
\begin{align}\label{eq.Dt}
  D_tk_{a_1\cdots a_r}=\frac{\D x^{i_1}}{\D y^{a_1}}\cdots \frac{\D x^{i_r}}{\D y^{a_r}}\left(D_t w_{i_1\cdots i_r}+\frac{\D v^\ell}{\D x^{i_1}}w_{\ell\cdots i_r}+\cdots +\frac{\D v^\ell}{\D x^{i_r}}w_{i_1\cdots\ell}\right).
\end{align}

We now recall a result of \cite{CL00} concerning time derivatives of the change of
coordinates and commutators between time derivatives and space derivatives: 

\begin{lemma}[{\cite[Lemma 2.1]{CL00}}] \label{lem.CL00lem2.1}
  Let $x=f_t(y)$ be the change of variables given by \eqref{trajectory}, and let $g_{ab}$ be the metric given by \eqref{metric1}. Let $v_i=\delta_{ij}v^j=v^i$, and set
  \begin{align}\label{eq.CL00lem2.1.1}
  u_a(t,y)=v_i(t,x)\frac{\D x^i}{\D y^a}, \;  u^a=g^{ab}u_b,\;
    h_{ab}=\frac{1}{2}D_t g_{ab},\;  h^{ab}=&g^{ac}h_{cd}g^{db}.
  \end{align}
  Then  
  \begin{align}\label{eq.CL00lem2.1.3}
    D_t\frac{\D x^i}{\D y^a}=\frac{\D x^k}{\D y^a} \frac{\D v^i}{\D x^k},\quad D_t\frac{\D y^a}{\D x^i}=-\frac{\D y^a}{\D x^k}\frac{\D v^k}{\D x^i},
  \end{align}
  and
  \begin{align}\label{eq.CL00lem2.1.4}
    D_t g_{ab}=\nb_a u_b+\nb_b u_a,\quad D_t g^{ab}=-2h^{ab},  \quad D_t d\mu_g=g^{ab}h_{ab}d\mu_g,
  \end{align}
  where $d\mu_g$ is the Riemannian volume element on $\Omega$ in the metric $g$.
\end{lemma}

We also recall from \cite{CL00} the estimates of commutators between the material derivative $D_t$ and space derivatives $\D_i$ and covariant derivatives $\nb_a$:

\begin{lemma}[\mbox{\cite[Lemma 2.3]{CL00}}]
  Let $\D_i$ be given by \eqref{eq.Di}. Then
  \begin{align}\label{eq.DtDi}
    [D_t,\D_i]=-(\D_i v^k)\D_k.
  \end{align}
  Furthermore,
  \begin{align}\label{eq.DtDir}
    [D_t,\D^r]=\sum_{s=0}^{r-1}-\left(r\atop s+1\right)(\D^{1+s}v)\cdot \D^{r-s},
  \end{align}
  where the symmetric dot product is defined to be in components
  \begin{align}
    \left((\D^{1+s}v)\cdot \D^{r-s}\right)_{i_1\cdots i_r}=\frac{1}{r!}\sum_{\sigma\in \Sigma_r}\left(\D_{i_{\sigma_1}\cdots i_{\sigma_{1+s}}}^{1+s} v^k\right)\D_{ki_{\sigma_{s+2}}\cdots i_{\sigma_r}}^{r-s},
  \end{align}
  and $\sum_r$ denotes the collection of all permutations of $\{1,2,\cdots, r\}$.
\end{lemma}

\begin{lemma}\label{lem.CL00lem2.4}
  Let $T_{a_1\cdots a_r}$ be a $(0,r)$ tensor. Then
  \begin{align}\label{eq.CL00lem2.4.1}
    [D_t,\nb_a]T_{a_1\cdots a_r}=-(\nb_{a_1}\nb_a u^d)T_{da_2\cdots a_r}-\cdots -(\nb_{a_r}\nb_a u^d)T_{a_1\cdots a_{r-1}d}.
  \end{align}
  In particular, 
  \begin{align}
  [D_t,g^{ab}\nb_a]T_b=&-2h^{ab}\nb_a T_b-(\Delta u^e)T_e,\label{eq.CL002.4.2}\\
      [D_t,g^{ac}g^{bd}\nb_a]T_{cd}=&-\Delta u_eT^{eb}-g^{bd}\nb_d\nb_a u_e T^{ae}-g^{ae}\nb_e u_f\nb_a T^{fb}\no\\
      &-g^{be}\nb_e u_f\nb_aT^{af}-\nb_f u^a\nb_a T^{fb}-\nb_fu^b\nb_a T^{af}.\label{eq.dtnb}
  \end{align}
  If $\nabla$ denotes the first order covariant derivative,  $\Delta=g^{cd}\nb_c\nb_d$ is the Laplacian operator under the Lagrangian coordinates and $q$ is a function, then
  \begin{align}\label{eq.Dtpcommu}
    [D_t,\nb]q=0, \quad
    [D_t,\Delta]q=-2h^{ab}\nb_a\nb_b q-(\Delta u^e)\nb_e q,
  \end{align}
 and
  \begin{align}\label{eq.commutator}
    [D_t,\nb^r]q=\sum_{s=1}^{r-1}-\left(r\atop s+1\right)(\nb^{s+1}u)\cdot \nb^{r-s} q,
  \end{align}
  where the symmetric dot product is defined as, in components,
  \begin{align}
    \left((\nb^{s+1}u)\cdot \nb^{r-s} q\right)_{a_1\cdots a_r}=\frac{1}{r!}\sum_{\sigma\in\Sigma_r}\left(\nb_{a_{\sigma_1}\cdots a_{\sigma_{s+1}}}^{s+1} u^d\right)\nb_{da_{\sigma_{s+2}}\cdots a_{\sigma_r}}^{r-s}q.
  \end{align}
\end{lemma}

\begin{proof}
 From \eqref{eq.CL00lem2.1.4} and \eqref{eq.CL00lem2.4.1}, it follows that
\begin{align*}
&[D_t,g^{ac}g^{bd}\nb_a]T_{cd}\\  
&=D_tg^{ac}g^{bd}\nb_aT_{cd}+g^{ac}D_tg^{bd}\nb_aT_{cd}+g^{ac}g^{bd}D_t\nb_aT_{cd}-g^{ac}g^{bd}\nb_aD_tT_{cd}\\
&=-2h^{ac}g^{bd}\nb_aT_{cd}-2g^{ac}h^{bd}\nb_aT_{cd}-g^{ac}g^{bd}(\nb_c\nb_au^eT_{ed}+\nb_d\nb_au^eT_{ce})\\
&=-\Delta u_eT^{eb}-g^{bd}\nb_d\nb_a u_e T^{ae}-g^{ae}\nb_e u_f\nb_a T^{fb}\\
&\quad -g^{be}\nb_e u_f\nb_aT^{af}-\nb_f u^a\nb_a T^{fb}-\nb_fu^b\nb_a T^{af}.
\end{align*}
For other identities, one can see the proofs in \cite[Lemma 2.4]{CL00} and \cite[Lemma 2.3]{HLarma}, and we omit the details.
\end{proof}

Let
$$\F_{ab}=F_{ij}\frac{\D x^i}{\D y^a}\frac{\D x^j}{\D y^b}, \quad
\F^{ab}=g^{ac}g^{bd}\F_{cd}, \quad |\F|^2=\F_{ab}\F^{ab}.$$
Then, it follows  from \eqref{eq.norminv} that
\begin{align}\label{eq.beta}
  |\F|^2=|F|^2 \quad \text{and} \quad \quad F_{ij}=\frac{\D y^a}{\D x^i}\frac{\D y^b}{\D x^j}\F_{ab}.
\end{align}
From \eqref{eq.CL00lem2.1.1}, \eqref{eq.beta}, \eqref{eq.CL00lem2.1.3} and \eqref{eq.covtensor}, it follows that
\begin{align*}
D_t u_a=&\detail{D_t\left(v_i\frac{\D x^i}{\D y^a}\right) =\frac{\D x^i}{\D y^a}D_t v_i+v_iD_t\frac{\D x^i}{\D y^a}\\
  =&-\nb_a p +\frac{\D y^b}{\D x^l}\frac{\D y^c}{\D x^k}\nb_c\F_{ab} \delta^{jk}\frac{\D y^e}{\D x^j} \delta^{ml}\frac{\D y^f}{\D x^m}\F_{ef}+ u_b\delta^{ij}\frac{\D y^b}{\D x^j}\frac{\D y^c}{\D x^i} \nb_a u_c\\
  =&}-\nb_a p +\nb_c\F_{ab}\F^{cb}+u^c\nb_a u_c.
\end{align*}
Similarly,  
\begin{align*}
  D_t\F_{ab}=&\detail{D_t\left(F_{ij}\frac{\D x^i}{\D y^a}\frac{\D x^j}{\D y^b}\right)=\frac{\D x^i}{\D y^a}\frac{\D x^j}{\D y^b}D_tF_{ij}+F_{ij}\left(D_t\frac{\D x^i}{\D y^a}\frac{\D x^j}{\D y^b}+\frac{\D x^i}{\D y^a}D_t\frac{\D x^j}{\D y^b}\right)\\
  =&\nb_du_a\frac{\D y^d}{\D x^k}\frac{\D y^c}{\D x^l}\delta^{kl}\F_{cb}+\F_{cb}\frac{\D y^c}{\D x^i}\nb_a u_e\frac{\D y^e}{\D x^j}\delta^{ij}+\F_{ac}\frac{\D y^c}{\D x^j}\nb_bu_d\frac{\D y^d}{\D x^i}\delta^{ij}\\
  =&}g^{dc}\nb_du_a \F_{cb}+\F_{cb}\nb_au^c+\F_{ac}\nb_bu^c.
\end{align*}
Then, we can rewrite the system \eqref{eld1} in the Lagrangian coordinates as
\begin{subequations}\label{eld41}
\begin{align}
    &D_tu_a+\nb_a p =\nb_c\F_{ab}\F^{cb}+u^c\nb_a u_c&&\quad \text{ in } [0,T]\times \Omega, \label{eld41.1}\\
    &D_t\F_{ab}=g^{dc}\nb_du_a \F_{cb}+\F_{cb}\nb_au^c+\F_{ac}\nb_bu^c &&\quad \text{ in } [0,T]\times \Omega,\label{eld41.2}\\
    &\nb_a u^a=0,\quad \nb_a\F^{ab}=0&&\quad \text{ in } [0,T]\times \Omega,\label{eld41.3}\\
    &p=0, \quad  N^a\F_{ab} =0 &&\quad \text{ on } [0,T]\times \D\Omega.\label{eld41.4}
\end{align}
\end{subequations}

From \eqref{energy.phys}, we also have the conserved energy
\begin{align}
E_0(t)=\int_{\Omega}\left(\frac{1}{2}\abs{u(t,y)}^2+\frac{1}{2}\abs{\F(t,y)}^2\right)d\mu_g.
\end{align}

We note that if
\begin{align}\label{eq.beta0}
     |\nb u(t,y)|\ls C\quad  \text{in } [0,T]\times\overline{\Omega},
\end{align}
and $\dv v=0$ in $[0,T]\times\Omega$, then  the divergence free property of $\F^\top$, i.e., $\dv\F^\top=0$, is preserved for all times under the Lagrangian coordinates or in view of the material derivative, i.e., $D_t\dv\F^\top=0$. Indeed, from \eqref{eq.dtnb} and Lemma \ref{lem.CL00lem2.1}, the divergence of $\eqref{eld41.2}$ gives
  \begin{align*}
    D_t\nb_a\F^{ab}=\detail{&D_t(g^{ac}g^{bd}\nb_a\F_{cd})=[D_t,g^{ac}g^{bd}\nb_a]\F_{cd}+g^{ac}g^{bd}\nb_a D_t \F_{cd}\\
    =&-\Delta u_e\F^{eb}-g^{bd}\nb_d\nb_a u_e \F^{ae}-g^{ae}\nb_e u_f\nb_a \F^{fb}-g^{be}\nb_e u_f\nb_a\F^{af}\\
    &-\nb_f u^a\nb_a \F^{fb}-\nb_fu^b\nb_a \F^{af}+\nb_f\nb_au^a\F^{fb} +\nb_fu^a\nb_a\F^{fb}\\
    &+g^{ac}\nb_a\F^{eb}\nb_c u_e+\F^{eb}\Delta u_e+g^{bd}\nb_a\F^{ae}\nb_d u_e+g^{bd}\F^{ae}\nb_a\nb_d u_e\\
    =&}-\nb_fu^b\nb_a \F^{af},
  \end{align*}
which implies, by the Gronwall inequality and  the identity $\big|D_t|f|\big|=\abs{D_t f}$, that
\begin{align*}
\abs{\nb_a\F^{ab}(t,y)}\ls e^{Ct}\abs{\nb_a\F^{ab}(0,y)}=0.
\end{align*}
Moreover,   $N\cdot \F^\top=0$ is also preserved for all time  $t$ in the lifespan $[0,T]$, that is, we have $N^a\F_{ab}=0$ on $[0,T]\times\D\Omega$ if $N\cdot \F^\top=0$ on the boundary $\D\Omega$ at initial time. In fact, we have, from \eqref{eld41.2},  Lemmas \ref{lem.CL00lem2.1} and \ref{lem.CL00lem3.9}, that
\begin{align*}
  D_t(N^a\F_{ab})=\detail{&D_t(g^{ac}N_c\F_{ab} )=D_t g^{ac}N_c\F_{ab}+g^{ac}D_t N_c\F_{ab}+N^aD_t\F_{ab}\\
  =&-2h^{ac}N_c\F_{ab}+g^{ac}h_{NN}N_c\F_{ab}+N^a(g^{dc}\nb_du_a \F_{cb}+\F_{cb}\nb_au^c+\F_{ac}\nb_bu^c)\\
  =&-g^{ad}\nb_du^cN_c\F_{ab}-\nb_eu^aN^e\F_{ab}+h_{NN}N^a\F_{ab}+g^{dc}\nb_du_a N^a\F_{cb}\\
  &+\F_{cb}\nb_au^cN^a+\F_{ac}\nb_bu^cN^a\\
  =&}h_{NN}N^a\F_{ab}+\nb_bu^cN^a\F_{ac},
\end{align*}
which yields similarly that
\begin{align*}
  |(N^a\F_{ab})(t,y)|\ls e^{Ct} |(N^a\F_{ab}(0,y)|=0.
\end{align*}

\bigskip

\section{The First Order Energy Estimates}\label{sec.1order}

In this section, we prove the first order energy estimate.
From \eqref{eq.CL00lem2.4.1} and \eqref{eld41.1},
we get
\begin{align*}
  D_t(\nb_b u_a)+\nb_b\nb_a p 
  =&-(\nb_a\nb_b u^d) u_d+\nb_b(\nb_c\F_{ad}\F^{cd})+\nb_b(u^c\nb_a u_c)\\
  =&\nb_b u^c\nb_au_c+\nb_b\nb_c\F_{ad}\F^{cd}+\nb_c\F_{ad}\nb_b\F^{cd}.
\end{align*}
From \eqref{eq.CL00lem2.4.1} and \eqref{eld41.2}, we obtain
\begin{align*}
D_t(\nb_c\F_{ab}) 
  =&-(\nb_a\nb_c u^d)\F_{db}-(\nb_b\nb_c u^d)\F_{ad}+g^{de}\nb_c\nb_d u_a\F_{eb}+g^{de}\nb_d u_a\nb_c\F_{eb}\\
  &+\nb_c\F_{db}\nb_a u^d+\F_{db}\nb_c\nb_a u^d+\nb_c\F_{ad}\nb_b u^d+\F_{ad}\nb_b\nb_c u^d\\
  =&g^{de}\nb_c\F_{eb}(\nb_d u_a+\nb_a u_d)+\nb_c\F_{ad}\nb_b u^d+g^{de}\F_{eb}\nb_c\nb_d u_a.
\end{align*}
Thus, we have
\begin{align}
D_t(\nb_b u_a)+&\nb_b\nb_a p
=\nb_b u^c\nb_au_c+\nb_b\nb_c\F_{ad}\F^{cd}+\nb_c\F_{ad}\nb_b\F^{cd},\label{eq.1energy1}\\
\intertext{and} 
D_t(\nb_c\F_{ab})=&g^{de}\nb_c\F_{eb}(\nb_d u_a+\nb_a u_d)+\nb_c\F_{ad}\nb_b u^d+g^{de}\F_{eb}\nb_c\nb_d u_a.\label{eq.1energy2}
\end{align}

We now derive the material derivative of $g^{bd}\gamma^{ae}\nb_a u_b\nb_e u_d$. From \eqref{eq.CL00lem2.1.4}, \eqref{eq.CL00lem2.1.1} and \eqref{eq.CL00lem3.9.1}, we have
\begin{align} \label{eq.1energy3}
  \detail{&}D_t(g^{bd}\gamma^{ae}\nb_a u_b\nb_e u_d)\no\\
  =&(D_t g^{bd})\gamma^{ae}\nb_a u_b\nb_e u_d +g^{bd}(D_t\gamma^{ae})\nb_a u_b\nb_e u_d +2g^{bd}\gamma^{ae}(D_t\nb_a u_b)\nb_e u_d\no\\
  =&-2\gamma^{ae}\gamma^{fc}\nb_eu_f\nb_a u^d\nb_c u_d-2\gamma^{ae}\nb_e u^b\nb_a\nb_b p \no\\
  &+2\gamma^{ae}\nb_e u^b(\nb_a\nb_c\F_{bf}\F^{cf}+\nb_c\F_{bf}\nb_a\F^{cf}),
\end{align}
 and from \eqref{eq.1energy2} it follows that
\begin{equation}\label{eq.1energy4}
\begin{split}
  &D_t(g^{bd}g^{cf}\gamma^{ae}\nb_a \F_{bc}\nb_e \F_{df}) \\
  =&-2g^{bq}h_{qs}g^{sd} g^{cf}\gamma^{ae}\nb_a \F_{bc}\nb_e \F_{df}-2g^{bd}g^{cq}h_{qs}g^{sf}\gamma^{ae}\nb_a \F_{bc}\nb_e \F_{df} \\
  &-2g^{bd}g^{cf}\gamma^{aq}h_{qs}\gamma^{se}\nb_a \F_{bc}\nb_e \F_{df}+2g^{bd}g^{cf}\gamma^{ae}D_t\nb_a \F_{bc}\nb_e \F_{df} \\
  =&-2\nb_qu_s\gamma^{aq}\gamma^{se}\nb_a \F^{df}\nb_e \F_{df}+2\nb_qu^d\gamma^{ae}\nb_a \F^{qf}\nb_e \F_{df}  \\   &+2\gamma^{ae}\nb_e \F_{bc}\F^{qc}\nb_a\nb_qu^b.
\end{split}\end{equation}
Combining \eqref{eq.1energy3} with \eqref{eq.1energy4} yields
\begin{align}\label{eq.1energy1.u}
  &D_t\left(g^{bd}\gamma^{ae}\nb_a u_b\nb_e u_d+g^{bd}g^{cf}\gamma^{ae}\nb_a \F_{bc}\nb_e \F_{df}\right)\no\\
  =&-2\gamma^{ae}\gamma^{fc}\nb_eu_f\nb_a u^d\nb_c u_d-2\gamma^{ae}\gamma^{fc}\nb_eu_f\nb_a \F^{ds}\nb_c \F_{ds}\no\\
  &-2\nb_b(\gamma^{ae}\nb_e u^b\nb_a p-\gamma^{ae}\nb_e u^c\nb_a\F_{cf}\F^{bf})\no\\
  &+2(\nb_b\gamma^{ae})(\nb_e u^b\nb_a p -\nb_e u^c\nb_a\F_{cf}\F^{bf})\no\\
  &+2\gamma^{ae}\nb_e u^d\nb_c\F_{df}\nb_a\F^{cf}+2\gamma^{ae}\nb_cu^d\nb_e \F_{df}\nb_a \F^{cf}.
\end{align}

Then we calculate the material derivatives of $|\curl u|^2$ and $|\curl \F^\top|^2$ where   $\curl \F^\top$ is defined as
$$(\curl\F^\top)_{abc}:=\nb_a\F_{bc}-\nb_b\F_{ac}.$$ 
Indeed, one has
\begin{align*}
    D_t|\curl u|^2=&D_t\left(g^{ac}g^{bd}(\curl u)_{ab}(\curl u)_{cd}\right)\\
    =&2(D_t g^{ac})g^{bd}(\curl u)_{ab}(\curl u)_{cd}+4 g^{ac}g^{bd}(D_t\nb_a u_b)(\curl u)_{cd}\\
    =&-4g^{ae}g^{bd}\nb_eu^c(\curl u)_{ab}(\curl u)_{cd}\\
    &+4g^{ac}g^{bd}(\curl u)_{cd}(\nb_a\F^{ef} \nb_e\F_{bf}+\nb_a\nb_e\F_{bf}\F^{ef}),
\end{align*}
and
\begin{align*}
    D_t|\curl \F^\top|^2=&D_t\left(g^{ac}g^{bd}g^{ef}(\curl\F^\top)_{abe}(\curl\F^\top)_{cdf}\right)\\
    =&-4g^{aq}g^{bd}g^{ef}\nb_qu^c(\curl \F^\top)_{abe}(\curl \F^\top)_{cdf}\\
    &-2g^{ac}g^{bd}g^{eq}\nb_q u^f(\curl \F^\top)_{abe}(\curl \F^\top)_{cdf}\\
    &+4 g^{ac}(\curl \F^\top)_{cdf}\nb_a\F^{sf}\nb_su^d+4 g^{ac}g^{bd}(\curl \F^\top)_{cdf}\nb_a\F^{qf}\nb_bu_q\\
    &+4 g^{ac}g^{ef}(\curl \F^\top)_{cdf}\nb_a\F^{ds}\nb_eu_s+4 g^{ac}(\curl \F^\top)_{cdf}\F^{sf}\nb_a\nb_su^d.
\end{align*}
Thus, we have obtained
\begin{align}\label{eq.1energy1.curl}
  &D_t(|\curl u|^2+|\curl \F^\top|^2)\no\\
  =&-4g^{ae}g^{bd}\nb_eu^c(\curl u)_{ab}(\curl u)_{cd}+4g^{ac}g^{bd}(\curl u)_{cd}\nb_a\F^{ef} \nb_e\F_{bf}\no\\
  &-4g^{aq}g^{bd}g^{ef}\nb_qu^c(\curl \F^\top)_{abe}(\curl \F^\top)_{cdf}\no\\
  &-2g^{ac}g^{bd}g^{eq}\nb_q u^f(\curl \F^\top)_{abe}(\curl \F^\top)_{cdf}\no\\
  &+4 g^{ac}(\curl \F^\top)_{cdf}\nb_a\F^{sf}\nb_su^d+4 g^{ac}g^{bd}(\curl \F^\top)_{cdf}\nb_a\F^{qf}\nb_bu_q\no\\
  &+4 g^{ac}g^{ef}(\curl \F^\top)_{cdf}\nb_a\F^{ds}\nb_eu_s+4\nb_e\left[g^{ac}g^{bd}(\curl u)_{cd}\nb_a\F_{bf}\F^{ef}\right].
\end{align}

Define the first order energy as
\begin{align}\label{energy.1}
   E_1(t)=&\int_\Omega \left(g^{bd}\gamma^{ae}\nb_a u_b\nb_e u_d+g^{bd}g^{cf}\gamma^{ae}\nb_a \F_{bc}\nb_e \F_{df}\right)d\mu_g\no\\
   &+\int_\Omega\left(|\curl u|^2+|\curl \F^\top|^2\right)d\mu_g.
\end{align}
Recall the Gauss formula for $\Omega$ and $\D\Omega$:
\begin{align}\label{Gauss}
  \int_\Omega \nb_a w^a d\mu_g=\int_{\D\Omega} N_a w^a d\mu_\gamma, \quad \text{and} \quad \int_{\D\Omega}\bnb_a\bar{f}^a d\mu_\gamma =0
\end{align}
if $\bar{f}$ is tangential to $\D\Omega$ and $N_a$ denotes the unit conormal to $\D\Omega$. Then, we can establish the following estimate on the first order energy:

\begin{theorem}\label{thm.1energy}
  For any smooth solution of system \eqref{eld41} for $0\ls t\ls T$ satisfying
  \begin{align}
    |\nb p|\ls M, \quad |\nb u|\ls& M,  &&\text{in } [0,T]\times \Omega,\\
    |\theta|+|\nb u|+\frac{1}{\iota_0}\ls &K,&&\text{on } [0,T]\times \D\Omega, \label{eq.1energy1.1}
  \end{align}
one has,  for any $t\in[0,T]$,
  \begin{align}
    E_1(t)\ls 2e^{CMt}E_1(0)+C K^2\left(\vol\Omega+ E_0(0)\right)\left(e^{CMt}-1\right)
  \end{align}
  with some constant $C>0$ which depends only on the dimension $n$.
\end{theorem}

\begin{proof}
  It follows, from \eqref{eq.1energy1.u}, \eqref{eq.1energy1.curl} and Gauss' formula, that
  \begin{align}
    &\frac{d}{dt}E_1(t)=\int_\Omega D_t\left(g^{bd}\gamma^{ae}\nb_a u_b\nb_e u_d+g^{bd}g^{cf}\gamma^{ae}\nb_a \F_{bc}\nb_e \F_{df}\right)d\mu_g\no\\
    &\qquad\qquad\quad+\int_\Omega \left(g^{bd}\gamma^{ae}\nb_a u_b\nb_e u_d+g^{bd}g^{cf}\gamma^{ae}\nb_a \F_{bc}\nb_e \F_{df}\right)\tr h d\mu_g\no\\
    &\qquad\qquad\quad+\int_\Omega D_t\left(|\curl u|^2+|\curl \F^\top|^2\right)d\mu_g\no \\
    &\qquad\qquad\quad+\int_\Omega \left(|\curl u|^2+|\curl \F^\top|^2\right)\tr h d\mu_g\no\\
   =&-2\int_\Omega \gamma^{ae}\gamma^{fc}\nb_eu_f\nb_a u^d\nb_c u_dd\mu_g-2\int_\Omega \gamma^{ae}\gamma^{fc}\nb_eu_f\nb_a \F^{ds}\nb_c \F_{ds}d\mu_g\no\\
     &-2\int_{\D\Omega} N_b(\gamma^{ae}\nb_e u^b\nb_a p-\gamma^{ae}\nb_e u^c\nb_a\F_{cf}\F^{bf})d\mu_\gamma\label{energy.1.2}\\
     &+2\int_\Omega (\nb_b\gamma^{ae})(\nb_e u^b\nb_a p -\nb_e u^c\nb_a\F_{cf}\F^{bf})d\mu_g\label{energy.1.3}\\
     &+2\int_\Omega \gamma^{ae}\nb_e u^d\nb_c\F_{df}\nb_a\F^{cf}d\mu_g+2\int_\Omega \gamma^{ae}\nb_cu^d\nb_e \F_{df}\nb_a \F^{cf}d\mu_g\no\\
     &-4\int_\Omega g^{ae}g^{bd}\nb_eu^c(\curl u)_{ab}(\curl u)_{cd}d\mu_g+4\int_\Omega g^{ac}g^{bd}(\curl u)_{cd}\nb_a\F^{ef} \nb_e\F_{bf}d\mu_g\no\\
       &-4\int_\Omega g^{aq}g^{bd}g^{ef}\nb_qu^c(\curl \F^\top)_{abe}(\curl \F^\top)_{cdf}d\mu_g\no\\
       &-2\int_\Omega g^{ac}g^{bd}g^{eq}\nb_q u^f(\curl \F^\top)_{abe}(\curl \F^\top)_{cdf}d\mu_g\no\\
       &+4\int_\Omega  g^{ac}(\curl \F^\top)_{cdf}\nb_a\F^{sf}\nb_su^dd\mu_g+4\int_\Omega  g^{ac}g^{bd}(\curl \F^\top)_{cdf}\nb_a\F^{qf}\nb_bu_qd\mu_g\no\\
       &+4\int_\Omega  g^{ac}g^{ef}(\curl \F^\top)_{cdf}\nb_a\F^{ds}\nb_eu_sd\mu_g\no\\
       &+4\int_{\D\Omega} N_e\F^{ef}g^{ac}g^{bd}(\curl u)_{cd}\nb_a\F_{bf}d\mu_\gamma\label{energy.1.9}\\
    &+\int_\Omega \left(g^{bd}\gamma^{ae}\nb_a u_b\nb_e u_d+g^{bd}g^{cf}\gamma^{ae}\nb_a \F_{bc}\nb_e \F_{df}\right)\tr h d\mu_g\no\\
   &+\int_\Omega \left(|\curl u|^2+|\curl \F^\top|^2\right)\tr h d\mu_g.\no
  \end{align}
  Since $p=0$ on the boundary $\D\Omega$, it follows that $\bnb p=0$, i.e., $\gamma_a^d\nb_d p=0$, and then $\gamma^{ae}\nb_a p=g^{ce}\gamma_c^a\nb_a p=0$ on  $\D\Omega$. In addition,  $N\cdot\F^\top=0$ on $\D\Omega$. Thus, the integrals in \eqref{energy.1.2} and \eqref{energy.1.9} vanish.

 For the term \eqref{energy.1.3}, we first have from \eqref{2ndfundform} and \eqref{gammaauc},  
  $$ \theta_{ab}=\detail{(\delta_a^c-N_aN^c)\nb_c N_b=\nb_a N_b-N_a\nb_N N_b=}\nb_a N_b$$
  since  $\nb_N N=0$ in geodesic coordinates, and then
  \begin{align*}
    \nb_b\gamma^{ae}&=\nb_b (g^{ae}-N^aN^e)=-\nb_b(N^aN^e)\\
    &=-(\nb_b N^a)N^e-(\nb_b N^e)N^a
    =-\theta_b^a N^e-\theta_b^eN^a.
  \end{align*}
  Thus,  by the H\"older inequality, \eqref{eq.1energy1.1} and Lemma \ref{lem.CL00lem5.5}, one has
  \begin{align*}
   & |\eqref{energy.1.3}|\\ &\ls  CK\left(\norm{\nb u}_{L^2(\Omega)}\norm{\nb p}_{L^\infty(\Omega)}(\vol\Omega)^{1/2}+\norm{\nb u}_{L^\infty(\Omega)}\norm{\F}_{L^2(\Omega)}\norm{\nb \F}_{L^2(\Omega)}\right)\\
    &\ls CK M\left((\vol\Omega)^{1/2}+ E_0^{1/2}(0)\right)E_1^{1/2}(t).
  \end{align*}
  
  For other terms, we can use the H\"older inequality directly. Hence, we obtain
  \begin{align*}
   & \frac{d}{dt}E_1(t)\\ 
 &  \ls  CK M\left((\vol\Omega)^{1/2}+ E_0^{1/2}(0)\right)E_1^{1/2}(t)\\
    &\quad+C\norm{\nb u}_{L^\infty(\Omega)}\left(\norm{\nb u}_{L^2(\Omega)}^2+\norm{\nb \F}_{L^2(\Omega)}^2+\norm{\curl u}_{L^2(\Omega)}^2+\norm{\curl\F^\top}_{L^2(\Omega)}^2\right)\\
&    \ls C K M\left((\vol\Omega)^{1/2}+ E_0^{1/2}(0)\right)E_1^{1/2}(t)
    +CM E_1(t).
  \end{align*}
  By the Gronwall inequality, we have
  \begin{align*}
    E_1^{1/2}(t)\ls e^{CMt/2}E_1^{1/2}(0)+C K\left((\vol\Omega)^{1/2}+ E_0^{1/2}(0)\right)\left(e^{CMt/2}-1\right),
  \end{align*}
  which yields the desired estimate.
\end{proof}

\bigskip

\section{The General $r$-th Order Energy Estimates}\label{sec.rorder}

In this section, we establish the higher order energy estimates.
In view of \eqref{eq.Dt}, \eqref{eq.DtDir} and \eqref{eld11.1}, one has
\begin{align*}
  D_t\nb^r u_a=\detail{&D_t\nb_{a_1}\cdots \nb_{a_r}u_a=D_t\nb_{a_1}\cdots \nb_{a_{r-1}}\left(\frac{\D x^i}{\D y^a}\frac{\D x^{i_r}}{\D y^{a_r}}\frac{\D v_i}{\D x^{i_r}}\right)\\
  =&D_t\Big(\frac{\D x^i}{\D y^a}\frac{\D x^{i_1}}{\D y^{a_1}}\cdots\frac{\D x^{i_r}}{\D y^{a_r}}\frac{\D^r v_i}{\D x^{i_1}\cdots\D x^{i_r}}\Big)\\
  =&\frac{\D x^i}{\D y^a}\frac{\D x^{i_1}}{\D y^{a_1}}\cdots\frac{\D x^{i_r}}{\D y^{a_r}}\Bigg(D_t\frac{\D^r v_i}{\D x^{i_1}\cdots\D x^{i_r}}+\frac{\D v^l}{\D x^{i_1}}\frac{\D^r v_i}{\D x^{l}\cdots\D x^{i_r}}+\cdots\\
  &\qquad\qquad\qquad\qquad\qquad+\frac{\D v^l}{\D x^{i_r}}\frac{\D^r v_i}{\D x^{i_1}\cdots\D x^{l}}+\frac{\D v^l}{\D x^{i}}\frac{\D^r v_l}{\D x^{i_1}\cdots\D x^{x_r}}\Bigg)\\
  =&\frac{\D x^i}{\D y^a}\frac{\D x^{i_1}}{\D y^{a_1}}\cdots\frac{\D x^{i_r}}{\D y^{a_r}}\Big([D_t,\D^r]v_i+\D^rD_t v_i\Big)+\nb u\cdot\nb^ru_a+\nb_a u^c\nb^r u_c\\
  =}&-\nb^r\nb_a p+\nb_a u^c\nb^r u_c -\sum_{s=1}^{r-1}\left(r\atop s+1\right)(\nb^{1+s}u)\cdot \nb^{r-s}u_a\\
  &+\sum_{s=0}^r\left(r\atop s\right) \nb^s\F^{cb}\nb^{r-s}\nb_c \F_{ab},
\end{align*}
where
\begin{align*}
  \left(\nb^s\F^{cb}\nb^{r-s} \nb_c\F_{ab}\right)_{a_1\cdots a_r}=\sum_{\Sigma_r}\nb_{a_{\sigma_1}\cdots a_{\sigma_s}}^s \F^{cb} \nb_{a_{\sigma_{s+1}}\cdots a_{\sigma_r}}^{r-s}\nb_c\F_{ab}.
\end{align*}
Then, using $\dv \F^\top=0$, we obtain,  for $r\gs 2$,
\begin{align}\label{eq.r.u}
  D_t\nb^r u_a+\nb^r\nb_a p
  =&(\curl u)_{ac}\nb^r u^c+\sgn(2-r)\sum_{s=1}^{r-2}\left(r\atop s+1\right)(\nb^{1+s}u)\cdot \nb^{r-s}u_a\no\\
  &+\nb_c\left(\F^{cb}\nb^r\F_{ab}\right)+\sum_{s=1}^r\left(r\atop s\right) \nb^s\F^{cb}\nb^{r-s}\nb_c \F_{ab}.
\end{align}
Similarly, by $\dv\F^\top=0$ again, we have,  for $r\gs 2$,
\begin{align}\label{eq.nb.beta}
 D_t\nb^r \F_{ab}=&D_t\Big(\frac{\D x^i}{\D y^a}\frac{\D x^j}{\D y^b}\frac{\D x^{i_1}}{\D y^{a_1}}\cdots\frac{\D x^{i_r}}{\D y^{a_r}}\frac{\D^r F_{ij}}{\D x^{i_1}\cdots\D x^{i_r}}\Big)\no\\
 =&\frac{\D x^i}{\D y^a}\frac{\D x^j}{\D y^b}\frac{\D x^{i_1}}{\D y^{a_1}}\cdots\frac{\D x^{i_r}}{\D y^{a_r}}\Big(D_t\frac{\D^r F_{ij}}{\D x^{i_1}\cdots\D x^{i_r}}+\frac{\D v^l}{\D x^{i_1}}\frac{\D^r F_{ij}}{\D x^{l}\cdots\D x^{i_r}} +\cdots\no\\
 &\qquad+\frac{\D v^l}{\D x^{i_r}}\frac{\D^r F_{ij}}{\D x^{x_1}\cdots\D x^{l}}+\frac{\D v^l}{\D x^{i}}\frac{\D^r F_{lj}}{\D x^{x_1}\cdots\D x^{x_r}}+\frac{\D v^l}{\D x^{j}}\frac{\D^r F_{il}}{\D x^{x_1}\cdots\D x^{x_r}}\Big)\no\\
 =&\frac{\D x^i}{\D y^a}\frac{\D x^j}{\D y^b}\frac{\D x^{i_1}}{\D y^{a_1}}\cdots\frac{\D x^{i_r}}{\D y^{a_r}}\Big([D_t,\D^r]F_{ij}+\D^r D_tF_{ij}\Big)\no\\
 &+\nb u\cdot \nb^r\F_{ab}+\nb_au^c\nb^r\F_{cb}+\nb_bu^c\nb^r\F_{ac}\no\\
 =&\nb_au^c\nb^r\F_{cb}+\nb_bu^c\nb^r\F_{ac}-\nb^r u^c\nb_c\F_{ab}+\nb_c\left(g^{cd}\F_{db}\nb^r u_a\right)\no\\
  &-\sgn(2-r)\sum_{s=1}^{r-2}\left(r\atop s+1\right)(\nb^{1+s}u)\cdot \nb^{r-s}\F_{ab}\no \\
  &+\sum_{s=1}^{r}\left(r\atop s\right) g^{cd}\nb^s\F_{db}\nb^{r-s} \nb_c u_a.
\end{align}

From Lemmas \ref{lem.CL00lem2.1} and \ref{lem.CL00lem3.9}, and \eqref{eq.r.u}, it follows that
\begin{align}
  &D_t\left(g^{bd}\gamma^{af}\gamma^{AF}\nb_A^{r-1}\nb_a u_b\nb_F^{r-1}\nb_f u_d\right)\no\\
  \detail{=&(D_tg^{bd})\gamma^{af}\gamma^{AF}\nb_A^{r-1}\nb_a u_b\nb_F^{r-1}\nb_f u_d+rg^{bd}(D_t\gamma^{af})\gamma^{AF}\nb_A^{r-1}\nb_a u_b\nb_F^{r-1}\nb_f u_d\no\\
  &+2g^{bd}\gamma^{af}\gamma^{AF}D_t(\nb_A^{r-1}\nb_a u_b)\nb_F^{r-1}\nb_f u_d\no\\}
  =&-2\nb_c u_e\gamma^{af}\gamma^{AF}\nb_A^{r-1}\nb_a u^c\nb_F^{r-1}\nb_f u^e-2r\nb_c u_e\gamma^{ac}\gamma^{ef}\gamma^{AF}\nb_A^{r-1}\nb_a u^d\nb_F^{r-1}\nb_f u_d\no\\
  &-2\nb_b\left(\gamma^{af}\gamma^{AF}\nb_F^{r-1}\nb_f u^b\nb_A^{r-1}\nb_a p\right) \label{eq.r.4}\\
  &+2\nb_b\left(\gamma^{af}\gamma^{AF}\right)\nb_F^{r-1}\nb_f u^b\nb_A^{r-1}\nb_a p+2\gamma^{af}\gamma^{AF}\nb_F^{r-1}\nb_f u^b(\curl u)_{bc}\nb_A^{r-1}\nb_a u^c\no\\
  &+2\sgn(2-r)\gamma^{af}\gamma^{AF}\nb_F^{r-1}\nb_f u_d\sum_{s=1}^{r-2}\left(r\atop s+1\right)\left((\nb^{s+1}u)\cdot \nb^{r-s}u^d\right)_{Aa}\no\\
  &+2\gamma^{af}\gamma^{AF}\nb_F^{r-1}\nb_f u^d\nb_c\left(\F^{cb}\nb_{A}^{r-1}\nb_a\F_{db}\right)\label{eq.r.1}\\
  &+2\gamma^{af}\gamma^{AF}\nb_F^{r-1}\nb_f u^d\sum_{s=1}^r\left(r\atop s\right) \left(\nb^s\F^{cb}\nb^{r-s}\nb_c \F_{db}\right)_{Aa}.\no
\end{align}
Similarly, we have
\begin{align}
  &D_t\left(g^{bd}g^{ce}\gamma^{af}\gamma^{AF}\nb_A^{r-1}\nb_a \F_{bc}\nb_F^{r-1}\nb_f \F_{de}\right)\no\\
  \detail{=&D_t(g^{bd})g^{ce}\gamma^{af}\gamma^{AF}\nb_A^{r-1}\nb_a \F_{bc}\nb_F^{r-1}\nb_f \F_{de}+g^{bd}D_t(g^{ce})\gamma^{af}\gamma^{AF}\nb_A^{r-1}\nb_a \F_{bc}\nb_F^{r-1}\nb_f \F_{de}\no\\
  &+rD_t(\gamma^{af})\gamma^{AF}\nb_A^{r-1}\nb_a \F_{bc}\nb_F^{r-1}\nb_f \F^{bc}+2\gamma^{af}\gamma^{AF}D_t(\nb_A^{r-1}\nb_a \F_{bc})\nb_F^{r-1}\nb_f \F^{bc}\no\\}
  =&-2\nb_k u^m\gamma^{af}\gamma^{AF}\nb_A^{r-1}\nb_a \F^{kc}\nb_F^{r-1}\nb_f \F_{mc}-2\nb_k u^m\gamma^{af}\gamma^{AF}\nb_A^{r-1}\nb_a \F^{dk}\nb_F^{r-1}\nb_f \F_{dm}\no\\
  &-2r\nb_d u_e\gamma^{ad}\gamma^{ef}\gamma^{AF}\nb_A^{r-1}\nb_a \F^{bc}\nb_F^{r-1}\nb_f \F_{bc}+2\gamma^{af}\gamma^{AF}\nb_A^{r-1}\nb_a \F_{ec} \nb_bu^e\nb_F^{r-1}\nb_f\F^{bc}\no\\
  &+2\gamma^{af}\gamma^{AF}\nb_A^{r-1}\nb_a \F_{be} \nb_cu^e\nb_F^{r-1}\nb_f\F^{bc}-2\gamma^{af}\gamma^{AF}\nb_F^{r-1}\nb_f \F^{bc}\nb_A^{r-1}\nb_a u^e\nb_e\F_{bc}\no\\
  &+2\gamma^{af}\gamma^{AF}\nb_F^{r-1}\nb_f \F^{bc}g^{ed}\F_{dc}\nb_e\nb_{A}^{r-1}\nb_a u_b\label{eq.r.2}\\
  &+2\sgn(2-r)\gamma^{af}\gamma^{AF}\nb_F^{r-1}\nb_f \F^{bc}\sum_{s=1}^{r-2}\left(r\atop s+1\right)\left((\nb^{1+s}u)\cdot \nb^{r-s}\F_{bc}\right)_{Aa}\no\\
  &+2\gamma^{af}\gamma^{AF}\nb_F^{r-1}\nb_f \F^{bc}\sum_{s=1}^{r}\left(r\atop s\right)\left( g^{ed}\nb^s\F_{dc}\nb^{r-s} \nb_e u_b\right)_{Aa}.\no
\end{align}

For \eqref{eq.r.1} and \eqref{eq.r.2}, one has
\begin{align}
\eqref{eq.r.1}+\eqref{eq.r.2}=\detail{&2\gamma^{af}\gamma^{AF}\nb_A^{r-1}\nb_a u^b\nb_e\left(\F^{ec}\nb_{F}^{r-1}\nb_f\F_{bc}\right)\no\\
&+2\gamma^{af}\gamma^{AF}\nb_e\nb_{A}^{r-1}\nb_a u^b\F^{ec}\nb_F^{r-1}\nb_f \F_{bc}\no\\
=}&2\nb_e\left(\gamma^{af}\gamma^{AF}\nb_A^{r-1}\nb_a u^b\F^{ec}\nb_{F}^{r-1}\nb_f\F_{bc}\right)\label{eq.r.3}\\
&-2\nb_e(\gamma^{af}\gamma^{AF})\nb_A^{r-1}\nb_a u^b\F^{ec}\nb_{F}^{r-1}\nb_f\F_{bc}.\label{eq.r.5}
\end{align}
The boundary integral stemmed from the integration of \eqref{eq.r.3} over $\Omega$ will vanish since it involves the term $N_e\F^{ec}$ which is zero on the boundary. Since \eqref{energy.1.2}, especially the integral involving $p$, vanishes, we do not need the boundary integral in the first order energy $E_1(t)$.  However, the boundary integral derived from the integral of \eqref{eq.r.4} over $\Omega$ will be out of control for higher order energies. Thus, we have to include a boundary integral to overcome this difficulty.

Define the $r$-th order energy for an integer $r\gs 2$ as
\begin{align*}
  E_r(t)=&\int_\Omega g^{bd}\gamma^{af}\gamma^{AF}\nb_A^{r-1}\nb_a u_b\nb_F^{r-1}\nb_f u_d d\mu_g+\int_\Omega |\nb^{r-1}\curl u|^2 d\mu_g\\
  &+\int_\Omega g^{bd}g^{ce}\gamma^{af}\gamma^{AF}\nb_A^{r-1}\nb_a \F_{bc}\nb_F^{r-1}\nb_f \F_{de} d\mu_g+\int_\Omega |\nb^{r-1}\curl \F^\top|^2 d\mu_g\\
  &+\int_{\D\Omega} \gamma^{af}\gamma^{AF}\nb_A^{r-1}\nb_a  p\nb_F^{r-1}\nb_f  p\, \vartheta d\mu_\gamma,
\end{align*}
where $\vartheta=1/(-\nb_N p)$ as before.  Then, we have the following theorem.

\begin{theorem}\label{thm.renergy}
  For the integer $r\in \{2,\cdots,n+1\}$,   there exists a constant $T>0$ such that, for any smooth solution to system \eqref{eld41} for $0\ls t\ls T$ satisfying
  \begin{align}
   |\F|\ls& M_1 \quad\text{for } r=2,&&\text{in } [0,T]\times \Omega,\label{eq.2energy81}\\
   \quad |\nb p|\ls M, \quad |\nb u|\ls& M, \quad |\nb \F|\ls M,  &&\text{in } [0,T]\times \Omega,\label{eq.2energy8}\\
    |\theta|+1/\iota_0\ls &K,&&\text{on } [0,T]\times \D\Omega,\label{eq.2energy9}\\
    -\nb_N p\gs \eps>&0, &&\text{on } [0,T]\times \D\Omega,\label{eq.2energy91}\\
    |\nb^2p|+|\nb_ND_tp|\ls& L,&&\text{on } [0,T]\times \D\Omega,\label{eq.2energy92}
  \end{align}
  the following estimate holds for any $t\in[0,T]$,
  \begin{align}\label{eq.renergy}
  E_r(t)\ls e^{C_1t}E_r(0)+C_2\left(e^{C_1t}-1\right),
\end{align}
where the constants $C_1$ and $C_2$ depend on $K$, $K_1$, $M$, $M_1$, $L$, $1/\eps$, $\vol\Omega$, $E_0(0)$, $E_1(0)$, $\cdots$, and $E_{r-1}(0)$.
\end{theorem}

\begin{proof}
The derivative of $E_r(t)$ with respect to $t$  is
  \begin{align}
    \frac{d}{dt}E_r(t)=&\int_\Omega D_t\left(g^{bd}\gamma^{af}\gamma^{AF}\nb_A^{r-1}\nb_a u_b\nb_F^{r-1}\nb_f u_d\right) d\mu_g\label{eq.r.e1}\\
  &+\int_\Omega D_t\left(g^{bd}g^{ce}\gamma^{af}\gamma^{AF}\nb_A^{r-1}\nb_a \F_{bc}\nb_F^{r-1}\nb_f \F_{de}\right) d\mu_g\label{eq.r.e2}\\
    &+\int_\Omega D_t|\nb^{r-1}\curl u|^2 d\mu_g
+\int_\Omega D_t|\nb^{r-1}\curl \F^\top|^2 d\mu_g\label{eq.r.e3}\\
    &+\int_\Omega g^{bd}\gamma^{af}\gamma^{AF}\nb_A^{r-1}\nb_a u_b\nb_F^{r-1}\nb_f u_d \tr h d\mu_g\label{eq.r.e4}\\
    &+\int_\Omega |\nb^{r-1}\curl u|^2 \tr hd\mu_g
  +\int_\Omega |\nb^{r-1}\curl \F^\top|^2 \tr hd\mu_g\label{eq.r.e5}\\
  &+\int_\Omega g^{bd}g^{ce}\gamma^{af}\gamma^{AF}\nb_A^{r-1}\nb_a \F_{bc}\nb_F^{r-1}\nb_f \F_{de} \tr h d\mu_g\label{eq.r.e6}\\
  &+\int_{\D\Omega} D_t\left(\gamma^{af}\gamma^{AF}\nb_A^{r-1}\nb_a  p\nb_F^{r-1}\nb_f  p\right)\, \vartheta d\mu_\gamma\label{eq.r.e7}\\
   &+\int_{\D\Omega} \gamma^{af}\gamma^{AF}\nb_A^{r-1}\nb_a  p\nb_F^{r-1}\nb_f  p\left(\frac{\vartheta_t}{\vartheta}+\tr h-h_{NN}\right)\, \vartheta d\mu_\gamma.\label{eq.r.e8}
  \end{align}

\subsection*{Step 1:  Estimate the integrals \eqref{eq.r.e1}, \eqref{eq.r.e2} and \eqref{eq.r.e7}} \quad   \\
  From the previous derivations for the integrands in \eqref{eq.r.e1} and \eqref{eq.r.e2}, \eqref{eq.r.3}, \eqref{eq.r.5}  and
\begin{align*}
  &D_t\left(\gamma^{af}\gamma^{AF}\nb_A^{r-1}\nb_a  p\nb_F^{r-1}\nb_f  p\right)\\
  =&-2r\nb_c u_e\gamma^{ac}\gamma^{ef}\gamma^{AF}\nb_A^{r-1}\nb_a  p\nb_F^{r-1}\nb_f  p+2 \gamma^{af}\gamma^{AF}\nb_A^{r-1}\nb_a  p D_t\left(\nb_F^{r-1}\nb_f  p\right),
\end{align*}
we have
\begin{align}
  &\eqref{eq.r.e1}+\eqref{eq.r.e2}+\eqref{eq.r.e7}\no\\
  \ls& C\left(\norm{\nb u}_{L^\infty(\Omega)}+\norm{\nb \F}_{L^\infty(\Omega)}\right) E_r(t)\no\\
  &+C E_r^{1/2}(t)\sum_{s=1}^{r-2}\norm{\nb^{s+1} u}_{L^4(\Omega)}\left(\norm{\nb^{r-s} u}_{L^4(\Omega)}+\norm{\nb^{r-s} \F}_{L^4(\Omega)}\right)\label{eq.r.e111}\\
  &+C E_r^{1/2}(t)\sum_{s=2}^{r-1}\norm{\nb^s \F}_{L^4(\Omega)}\left(\norm{\nb^{r-s+1}u}_{L^4(\Omega)}+\norm{\nb^{r-s+1} \F}_{L^4(\Omega)}\right)\label{eq.r.e112}\\
  &+2\int_{\D\Omega}\gamma^{af}\gamma^{AF}\nb_{Aa}^{r} p \left(D_t\nb_{Ff}^{r}  p-\frac{1}{\vartheta}N_b\nb_{Ff}^{r} u^b\right) \vartheta d\mu_\gamma \label{eq.r.e11}\\
  &+2\int_{\Omega}\nb_b\left(\gamma^{af}\gamma^{AF}\right)\nb_F^{r-1}\nb_f u^b\nb_A^{r-1}\nb_a p d\mu_g\label{eq.r.e12}\\
  &+\int_{\D\Omega} N_e\gamma^{af}\gamma^{AF}\nb_A^{r-1}\nb_a u^b\F^{ec}\nb_{F}^{r-1}\nb_f\F_{bc} d\mu_\gamma\label{eq.r.e13}\\
  &-\int_{\Omega} \nb_e(\gamma^{af}\gamma^{AF})\nb_A^{r-1}\nb_a u^b\F^{ec}\nb_{F}^{r-1}\nb_f\F_{bc} d\mu_g.\label{eq.r.e14}
\end{align}

Since $N\cdot\F^\top=0$ on $\D\Omega$,  \eqref{eq.r.e13} vanishes.  From Lemma \ref{lem.CL00lemA.4}, we see that, for $\iota_1\gs 1/K_1$,  
\begin{align}\label{eq.CllemmaA.4}
  \norm{\F}_{L^\infty(\Omega)}\ls C\sum_{0\ls s\ls 2} K_1^{n/2-s} \norm{\nb^s \F}_{L^2(\Omega)}\ls C(K_1)\sum_{s=0}^2 E_s^{1/2}(t).
\end{align}
Using the H\"older inequality and the assumption \eqref{eq.2energy8}, we obtain for any integers $r\gs 3$,
\begin{align}
  \abs{\eqref{eq.r.e14}}\ls &CK\norm{\F}_{L^\infty(\Omega)}E_r(t)\ls C(K,K_1) \left(\sum_{s=0}^2 E_s^{1/2}(t)\right)E_r(t).
\end{align}
For $r=2$, by \eqref{eq.2energy81}, one has
\begin{align}
  \abs{\eqref{eq.r.e14}}\ls &CK\norm{\F}_{L^\infty(\Omega)}E_r(t)\ls C(K,M_1) E_r(t).
\end{align}

\subsubsection*{Step 1.1: Estimate \eqref{eq.r.e12}} \quad \\
From  H\"older's inequality, we have
\begin{align}\label{eq.up}
  \abs{\eqref{eq.r.e12}}\ls CKE_r^{1/2}(t)\norm{\nb^r  p}_{L^2(\Omega)}.
\end{align}

It follows from \eqref{eld41.1} and \eqref{eq.CL002.4.2} that
\begin{align}\label{eq.2e.p}
  \Delta p=-\nb_a u^b\nb_b u^a+g^{cb}\nb_a \F_{cd} \nb_b \F^{ad}.
\end{align}
Then,   for $r\gs 2$, 
\begin{align*}
  \nb^{r-2}\Delta  p=&-\sum_{s=0}^{r-2}\left(r-2\atop s\right) \nb^s\nb_a u^b\nb^{r-2-s}\nb_b u^a\\
  &+\sum_{s=0}^{r-2}\left(r-2\atop s\right)g^{cb}\nb^s\nb_a \F_{cd}\nb^{r-2-s}\nb_b \F^{ad}.
\end{align*}
In view of \eqref{eq.CllemmaA.4},  one has,  for $s\gs 0$,
\begin{align}\label{eq.r.einfbeta}
  \norm{\nb^s \F}_{L^\infty(\Omega)}\ls &C\sum_{\ell=0}^{2} K_1^{n/2-\ell}\norm{\nb^{\ell+s}\F}_{L^2(\Omega)}
  \ls C(K_1)\sum_{\ell=0}^{2} E_{s+\ell}^{1/2}(t),
\end{align}
and
\begin{align}\label{eq.r.einfu}
  \norm{\nb^s u}_{L^\infty(\Omega)}\ls C(K_1)\sum_{\ell=0}^{2} E_{s+\ell}^{1/2}(t).
\end{align}
From the H\"older inequality, \eqref{eq.r.einfbeta} and \eqref{eq.r.einfu}, we have, for $r\in\{3,4\}$,  
\begin{align}\label{eq.deltap34}
  &\norm{\nb^{r-2}\Delta p}_{L^2(\Omega)}\no \\
  &\ls  C\sum_{s=0}^{r-2} \norm{\nb^s\nb_a u^b\nb^{r-2-s}\nb_b u^a}_{L^2(\Omega)}\no\\
  &\quad+C\sum_{s=0}^{r-2}\norm{g^{cb}\nb^s\nb_a \F_{cd}\nb^{r-2-s}\nb_b \F^{ad}}_{L^2(\Omega)}\no\\
 & \ls C\norm{\nb u}_{L^\infty(\Omega)}\norm{\nb^{r-1} u}_{L^2(\Omega)}+C\norm{\nb \F}_{L^\infty(\Omega)}\norm{\nb^{r-1} \F}_{L^2(\Omega)}\no\\
  &\quad+(r-3)C\left(\norm{\nb^{2} u}_{L^\infty(\Omega)}\norm{\nb^{2} u}_{L^2(\Omega)}+\norm{\nb^{2} \F}_{L^\infty(\Omega)}\norm{\nb^{2} \F}_{L^2(\Omega)}\right)\no\\
 & \ls C(K_1)\sum_{\ell=1}^{r-1} E_\ell(t)+C(K_1)E_2^{1/2}(t)E_r^{1/2}(t).
\end{align}
For the case $r=2$, we have the following estimate from the assumption \eqref{eq.2energy8} and the H\"older inequality:
\begin{align}\label{eq.deltap2}
  \norm{\Delta p}_{L^2(\Omega)}\ls&C\norm{\nb u}_{L^2(\Omega)}\norm{\nb u}_{L^\infty(\Omega)}+C\norm{\nb \F}_{L^2(\Omega)}\norm{\nb \F}_{L^\infty(\Omega)}
  \ls CME_1^{1/2}(t),
\end{align}
which is a lower order energy term.
Hence, from \eqref{eq.CL00prop5.8.2}, \eqref{eq.deltap34} and \eqref{eq.deltap2}, we have for any real number $\delta_r>0$,
\begin{align}\label{eq.est.nbrp}
\norm{\nb^r  p}_{L^2(\Omega)}
  \ls &\delta_r\norm{\Pi \nb^r  p}_{L^2(\D\Omega)} +C(1/\delta_r,K,\vol\Omega)\sum_{s\ls r-2}\norm{\nb^s\Delta  p}_{L^2(\Omega)}\no\\
  \ls&\delta_r\norm{\Pi \nb^r  p}_{L^2(\D\Omega)}+C(1/\delta_r,K,K_1,M,\vol\Omega)\sum_{\ell=1}^{r-1} E_\ell(t)\no\\
  &+(r-2)C(1/\delta_r,K,K_1,M,\vol\Omega)E_2^{1/2}(t)E_r^{1/2}(t).
\end{align}

Next, we estimate the boundary terms. Because $p=0$ on the boundary $\D\Omega$, from \eqref{eq.CL00prop5.9.1}, we obtain for $r\gs 1$,
\begin{align}\label{eq.pinbrp}
    \norm{\Pi\nb^r p}_{L^2(\D\Omega)}\ls & C(K,K_1)\Big(\norm{\theta}_{L^\infty(\D\Omega)}+(r-2)\sum_{k\ls r-3} \norm{\bnb^k \theta}_{L^2(\D\Omega)}\Big)\no \\
    &\times\sum_{k\ls r-1}\norm{\nb^k p}_{L^2(\D\Omega)}.
  \end{align}
Due to \eqref{eq.CL004.20}, we have $\Pi\nb^2 p=\theta \nb_N p$. From \eqref{eq.2energy91}, \eqref{eq.2energy9}, \eqref{eq.CL00lemA.7.1}, \eqref{eq.2energy8} and \eqref{eq.est.nbrp}, we obtain
\begin{align}
  \norm{\theta}_{L^2(\D\Omega)}=&\norm{\frac{\Pi\nb^2 p}{\nb_N p}}_{L^2(\D\Omega)}\ls \frac{1}{\eps} \norm{\Pi\nb^2 p}_{L^2(\D\Omega)},
  \end{align}
  and 
  \begin{align}
  \norm{\Pi\nb^2 p}_{L^2(\D\Omega)}\ls& \norm{\theta}_{L^\infty(\D\Omega)}\norm{\nb p}_{L^2(\D\Omega)}\ls C(K,\vol\Omega)\left(\norm{\nb^2 p}_{L^2(\Omega)}+\norm{\nb p}_{L^2(\Omega)}\right)\no\\
  \ls& C(K,\vol\Omega)\delta_2\norm{\Pi\nb^2 p}_{L^2(\D\Omega)} +C(K,\vol\Omega)(\vol\Omega)^{1/2} M\no\\
  &+C(1/\delta_2,K,K_1,M,\vol\Omega)E_1(t),\label{eq.pinb2p}
\end{align}
where the first term on the right hand side of \eqref{eq.pinb2p} can be absorbed by the left hand side if we take $\delta_2$ small such that  $C(K,\vol\Omega)\delta_2\ls 1/2$. Then,  
\begin{align}
  \norm{\Pi\nb^2 p}_{L^2(\D\Omega)}+\norm{\nb^2 p}_{L^2(\Omega)}\ls&C(K,K_1,M,\vol\Omega)(1+E_1(t)),\label{eq.nb2p}\\
  \norm{\theta}_{L^2(\D\Omega)}\ls&C(K,K_1,M,\vol\Omega,1/\eps)(1+E_1(t)).\label{eq.theta}
\end{align}
From Theorem \ref{thm.1energy}, there exists a constant $T>0$ such that $E_1(t)$ can be controlled by the initial energy $E_1(0)$ for $t\in[0,T]$, e.g., $E_1(t)\ls 2E_1(0)$.
Then, from \eqref{eq.pinbrp}, \eqref{eq.theta}, \eqref{eq.2energy8} and \eqref{eq.nb2p}, we get
\begin{align*}
  \norm{\Pi\nb^3 p}_{L^2(\D\Omega)}\ls &C(K,K_1)\left(K+\norm{\theta}_{L^2(\D\Omega)}\right) \sum_{k\ls 2}\norm{\nb^k p}_{L^2(\D\Omega)}\no\\
  \ls& C(K,K_1,M,\vol\Omega,1/\eps,E_1(0))\norm{\nb^3 p}_{L^2(\Omega)} \no\\ &+C(K,K_1,M,\vol\Omega,1/\eps,E_1(0)).
\end{align*}
It follows from \eqref{eq.est.nbrp} that
\begin{align*}
  \norm{\nb^3  p}_{L^2(\Omega)}
  \ls& \delta_3 C(K,K_1,M,\vol\Omega,1/\eps,E_1(0))\norm{\nb^3 p}_{L^2(\Omega)}\\
  &+\delta_3 C(K,K_1,M,\vol\Omega,1/\eps,E_1(0)) \no\\
  &+C(1/\delta_3,K,K_1,M,\vol\Omega) (E_1(t)+E_2(t))\no\\
  &+C(1/\delta_3,K,K_1,M,\vol\Omega)E_2^{1/2}(t)E_3^{1/2}(t),
\end{align*}
which, if we take $\delta_3>0$ so small that $\delta_3 C(K,K_1,M,\vol\Omega,1/\eps,E_1(0))\ls 1/2$,
implies  
\begin{align}\label{eq.nb3p}
  \norm{\nb^3 p}_{L^2(\Omega)}
  \ls&  C(K,K_1,M,\vol\Omega,1/\eps,E_1(0))+C(K,K_1,M,\vol\Omega)\sum_{\ell=1}^2 E_\ell(t)\no\\
  &+C(K,K_1,M,\vol\Omega)E_2^{1/2}(t)E_3^{1/2}(t),
\end{align}
and thus  
\begin{align*}
  \norm{\Pi\nb^3  p}_{L^2(\D\Omega)}
  \ls&  C(K,K_1,M,\vol\Omega,1/\eps,E_1(0))\left(1+\sum_{\ell=1}^2 E_\ell(t)+E_2^{1/2}(t)E_3^{1/2}(t)\right).
\end{align*}
Because
\begin{align*}
  \bnb_b\nb_N  p=&\gamma_b^d\nb_d (N^a\nb_a  p)
  =\detail{(\delta_b^d-N_bN^d)((\nb_dN^a)\nb_a  p+N^a\nb_d\nb_a p)\\
  =&}\theta_b^a\nb_a  p+N^a\nb_b\nb_a p-N_bN^d(\theta_d^a\nb_a p+N^a\nb_d\nb_a p),
\end{align*}
we have from \eqref{eq.CL00lemA.7.1} that
\begin{align*}
  \norm{\bnb\nb_N p}_{L^2(\D\Omega)}
  \ls& C\norm{\theta}_{L^\infty(\D\Omega)}\norm{\nb  p}_{L^2(\D\Omega)}+C\norm{\nb^2 p}_{L^2(\D\Omega)}\\
  \ls&C(K,\vol\Omega)\left(\norm{\nb^3  p}_{L^2(\Omega)}+\norm{\nb^2  p}_{L^2(\Omega)}+\norm{\nb  p}_{L^2(\Omega)}\right)\\
  \ls& C(K,K_1,M,\vol\Omega,1/\eps,E_1(0))+C(K,K_1,M,\vol\Omega)\sum_{\ell=1}^2 E_\ell(t)\\
  &+C(K,K_1,M,\vol\Omega)E_2^{1/2}(t)E_3^{1/2}(t).
\end{align*}
Then, by \eqref{eq.CL004.21}, we have $$(\bnb\theta)\nb_N  p=\Pi\nb^3  p-3\theta\tilde{\otimes}\bnb\nb_N  p$$ and
\begin{align*}
  \norm{\bnb\theta}_{L^2(\D\Omega)}\ls& \frac{1}{\eps}\left(\norm{\Pi\nb^3 p}_{L^2(\D\Omega)}+C\norm{\theta}_{L^\infty(\D\Omega)}\norm{\bnb\nb_N p}_{L^2(\D\Omega)}\right)\no\\
  \ls&C(K,K_1,M,\vol\Omega,1/\eps,E_1(0))\Big(1+\sum_{\ell=1}^2 E_\ell(t)+E_2^{1/2}(t)E_3^{1/2}(t)\Big).
\end{align*}
From \eqref{eq.pinbrp} and \eqref{eq.CL00lemA.7.1},  
\begin{align}
  \norm{\Pi\nb^4 p}_{L^2(\D\Omega)}\ls &C(K,K_1)\left(K+\norm{\theta}_{L^2(\D\Omega)}+ \norm{\bnb \theta}_{L^2(\D\Omega)}\right)\sum_{k\ls 4}\norm{\nb^k p}_{L^2(\Omega)}.\label{eq.nb4p}
\end{align}

Thus, by \eqref{eq.est.nbrp}  we can absorb the highest order term $\norm{\nb^4 p}_{L^2(\Omega)}$ by the left hand side for $\delta_4>0$ small enough which is independent of the highest order energy $E_4(t)$, and  
\begin{align*}
  &\norm{\nb^4 p}_{L^2(\Omega)}+\norm{\Pi\nb^4 p}_{L^2(\D\Omega)}\\
  \ls& C(K,K_1,M,\vol\Omega,1/\eps,E_1(0))\Big(1+\sum_{\ell=1}^3 E_\ell(t)+E_2^{1/2}(t)E_4^{1/2}(t)\Big).
\end{align*}
Hence, from \eqref{eq.nb2p}, \eqref{eq.nb3p} and \eqref{eq.nb4p}, we have for $r\gs 2$,   
\begin{align*}
  \norm{\nb^r p}_{L^2(\Omega)}\ls &C(K,K_1,M,\vol\Omega,1/\eps,E_1(0))\Big(1+\sum_{\ell=1}^{r-1} E_\ell(t)+(r-2)E_2^{1/2}(t)E_r^{1/2}(t)\Big),
\end{align*}
which, from \eqref{eq.up}, yields  
\begin{align*}
  \abs{\eqref{eq.r.e12}}\ls& C(K,K_1,M,\vol\Omega,1/\eps,E_1(0))E_r^{1/2}(t)\Big(1+\sum_{\ell=1}^{r-1} E_\ell(t)+(r-2)E_2^{1/2}(t)E_r^{1/2}(t)\Big).
\end{align*}

\subsubsection*{Step 1.2: Estimate \eqref{eq.r.e11}} \quad \\
  The boundary condition $p=0$ on $\D\Omega$ implies $\gamma_b^a\nb_a  p=0$ on $\D\Omega$. Then we have, from \eqref{gammaauc} and $\vartheta=-1/\nb_N  p$,  
\begin{align}\label{eq.nunb}
  -\vartheta^{-1}N_b=&\nb_N  p N_b=N^a\nb_a  p N_b=\delta_b^a\nb_a p-\gamma_b^a\nb_a p
  =\nb_b  p.
\end{align}
From the H\"older inequality and \eqref{eq.nunb}, we get
\begin{align}
  \abs{\eqref{eq.r.e11}}\ls &C\norm{\vartheta}_{L^\infty(\D\Omega)}^{1/2} E_r^{1/2}(t)\norm{\Pi\left(D_t\left(\nb^{r}  p\right)-\vartheta^{-1}N_b\nb^{r} u^b\right)}_{L^2(\D\Omega)}\no\\
  = &C\norm{\vartheta}_{L^\infty(\D\Omega)}^{1/2} E_r^{1/2}(t)\norm{\Pi\left(D_t\left(\nb^{r}  p\right)+\nb^{r} u\cdot\nb  p\right)}_{L^2(\D\Omega)}.
\end{align}
It follows from \eqref{eq.commutator} that
\begin{align}\label{eq.Piterm1}
  D_t\nb^r  p+\nb^{r} u\cdot\nb  p=&[D_t,\nb^r] p+\nb^r D_t  p+\nb^{r} u\cdot\nb  p\no\\
  =&\sgn(2-r)\sum_{s=1}^{r-2}\left(r\atop s+1\right)(\nb^{s+1}u)\cdot \nb^{r-s}  p+\nb^r D_t  p.
\end{align}

Now, we consider   the last term in \eqref{eq.Piterm1}. From \eqref{eq.CL00prop5.9.1} and \eqref{eq.CL00lemA.7.1}, we have, for $2\ls r\ls 4$,   
\begin{align}\label{eq.est.pinbdtp}
  &\norm{\Pi\nb^rD_t p}_{L^2(\D\Omega)}\no\\
  \ls& C(K,K_1,\vol\Omega)\left(\norm{\theta}_{L^\infty(\D\Omega)}+(r-2)\sum_{k\ls r-3}\norm{\bnb^k\theta}_{L^2(\D\Omega)}\right) \sum_{k\ls r} \norm{\nb^k D_t p}_{L^2(\Omega)}.
\end{align}
It follows from \eqref{eq.CL00prop5.8.2} that
\begin{align}\label{eq.est.nbkdtp}
  \norm{\nb^r D_t p}_{L^2(\Omega)}
  \ls \delta \norm{\Pi\nb^r D_t p}_{L^2(\D\Omega)}+C(1/\delta, K,\vol\Omega)\sum_{s\ls r-2}\norm{\nb^s\Delta D_t  p}_{L^2(\Omega)}.
\end{align}
From \eqref {eq.Dtpcommu}, \eqref{eq.2e.p}, Lemma \ref{lem.CL00lem2.1}, \eqref{eq.1energy1}, \eqref{eq.1energy2} and \eqref{eld41},  it follows that
\begin{align*}
  \Delta D_t  p=&2h^{ab}\nb_a\nb_b p+(\Delta u^e)\nb_e  p-D_t(g^{bd}g^{ac}\nb_a u_d\nb_b u_c)+D_t(g^{cb}g^{ae}g^{df}\nb_a \F_{cd} \nb_b \F_{ef})\\
  =&2h^{ab}\nb_a\nb_b p+(\Delta u^e)\nb_e  p-2D_t(g^{bd})\nb_a u_d\nb_b u^a-2 g^{bd}D_t(\nb_a u_d)\nb_b u^a\\
  &+2 D_t(g^{cb})\nb_a \F_{cd} \nb_b \F^{ad} +g^{cb}g^{ae}D_t(g^{df})\nb_a \F_{cd} \nb_b \F_{ef}+2 g^{cb} D_t(\nb_a \F_{cd}) \nb_b \F^{ad}\\
  =&2g^{ac}\nb_c u^b\nb_a\nb_b p+(\Delta u^e)\nb_e  p+2\nb_e u^b\nb_b u^a\nb_a u^e-2g^{ce}\nb_e u^b\nb_a\F_{cd}\nb_b\F^{ad}\\
  &-2\nb_du_f\nb_a\F^{bd}\nb_b\F^{af}+2g^{ce}\nb_c\F^{ad}\nb_a\F_{eb}\nb_d u^b-2g^{bd}\nb_b u^a \nb_a\nb_c \F_{de}\F^{ce}\\
  &+2g^{ce}\nb_b\F^{ad}\F_{ed}\nb_a\nb_c u^b.
\end{align*}
From \eqref{eq.r.einfbeta}, \eqref{eq.est.nbrp} and Lemma \ref{lem.CL00lemA.4}, it implies that,  for $s\ls 2$,
\begin{align*} 
&  \norm{\nb^s\Delta D_t p}_{L^2(\Omega)}\\
  \ls&C\norm{\nb  u}_{L^\infty(\Omega)}\norm{\nb^{s+2}  p}_{L^2(\Omega)}+s(s-1)C\norm{\nb^3 u}_{L^2(\Omega)}\norm{\nb^2  p}_{L^\infty(\Omega)}\no\\
  &+sC\norm{ \nb^2 u}_{L^4(\Omega)}\norm{\nb^{s+1}  p}_{L^4(\Omega)}+C\norm{\nb^{s+2} u}_{L^2(\Omega)}\norm{\nb   p}_{L^\infty(\Omega)}\no\\
  &+C\left(\norm{\nb u}_{L^\infty(\Omega)}\norm{\nb u}_{L^\infty(\Omega)}+\norm{\nb \F}_{L^\infty(\Omega)}\norm{\nb \F}_{L^\infty(\Omega)}\right) \norm{\nb^{s+1} u}_{L^2(\Omega)}\no\\
  &+s(s-1)C\norm{\nb u}_{L^\infty(\Omega)}\norm{\nb^2 u}_{L^4(\Omega)} \norm{\nb^2 u}_{L^4(\Omega)}\no\\
  &+C\norm{\nb u}_{L^\infty(\Omega)}\norm{\nb \F}_{L^\infty(\Omega)}\norm{\nb^{s+1} \F}_{L^2(\Omega)}\no\\
  &+sC\norm{\nb^2 u}_{L^4(\Omega)}\norm{\nb^2 \F}_{L^4(\Omega)} \left((s-1)\norm{\nb  \F}_{L^\infty(\Omega)}+\norm{\F}_{L^\infty(\Omega)}\right)\no\\
  &+s(s-1)C\norm{\nb u}_{L^\infty(\Omega)}\norm{\nb^2 \F}_{L^4(\Omega)} \norm{\nb^2  \F}_{L^4(\Omega)}\no\\
  &+C\norm{\nb u}_{L^\infty(\Omega)}\norm{\F}_{L^\infty(\Omega)} \norm{\nb^{s+2} \F}_{L^2(\Omega)}\no\\
  &+sC\norm{\nb^3 u}_{L^2(\Omega)}\norm{\F}_{L^\infty(\Omega)} \left((s-1)\norm{\nb^{2} \F}_{L^\infty(\Omega)}+\norm{\nb \F}_{L^\infty(\Omega)}\right)\no\\
  &+s(s-1)C\norm{\nb^3 \F}_{L^2(\Omega)}\norm{\F}_{L^\infty(\Omega)} \norm{\nb^{2} u}_{L^\infty(\Omega)}\no\\
  &+s(s-1)C\norm{\nb \F}_{L^\infty(\Omega)}\norm{\nb^2\F}_{L^4(\Omega)} \norm{\nb^{2} u}_{L^4(\Omega)}\no\\
  &+s(s-1)C\norm{\nb \F}_{L^\infty(\Omega)}\norm{\F}_{L^\infty(\Omega)} \norm{\nb^4 u}_{L^2(\Omega)}\no\\
  &+s(s-1)C\norm{\nb^2 \F}_{L^\infty(\Omega)}\norm{\F}_{L^\infty(\Omega)} \norm{\nb^3 u}_{L^2(\Omega)}.
\end{align*}
In view of Lemma~\ref{lem.CL00lemA.3} and \eqref{eq.r.einfu}, the following holds
\begin{align*}
  \norm{\nb^{s+1} u}_{L^4(\Omega)}\ls& C\norm{\nb^s u}_{L^\infty(\Omega)}^{1/2}\left(\sum_{\ell=0}^2\norm{\nb^{s+\ell} u}_{L^2(\Omega)}K_1^{2-\ell}\right)^{1/2}
  \ls C(K_1)\sum_{\ell=0}^2 E_{s+\ell}^{1/2}(t).
\end{align*}
We can estimate all the terms with $L^4(\Omega)$ norms in the same way in view of \eqref{eq.r.einfbeta}, \eqref{eq.r.einfu}, the similar estimate of $p$ and the assumptions. Hence, we obtain the bound which is linear with respect to the highest-order derivative or the highest-order energy $E_r^{1/2}(t)$, i.e.,
\begin{align}\label{eq.nbdtp}
  \norm{\nb^s\Delta D_t p}_{L^2(\Omega)}\ls& C(K,K_1,M,M_1,L,1/\eps,\vol\Omega,E_0(0))\Big(1+\sum_{\ell=0}^{r-1}E_\ell(t)\Big)\big(1+E_r^{1/2}(t)\big).
\end{align}
Therefore, by \eqref{eq.est.pinbdtp}, \eqref{eq.est.nbkdtp}, \eqref{eq.nbdtp} and for some small $\delta$   independent of $E_r(t)$, we obtain, by the induction argument for $r$,  
\begin{align}\label{eq.pinbrdtp}
  \norm{\Pi \nb^r D_t p}_{L^2(\D\Omega)}\ls &C(K,K_1,M,M_1,L,1/\eps,\vol\Omega,E_0(0))\no\\
   &\qquad\times\Big(1+\sum_{\ell=0}^{r-1}E_\ell(t)\Big)\big(1+E_r^{1/2}(t)\big).
\end{align}

For the estimate of \eqref{eq.Piterm1}, it only remains to estimate
\begin{align*}
    \norm{\Pi\left((\nb^{s+1}u)\cdot \nb^{r-s}  p\right)}_{L^2(\D\Omega)} \text{ for } 1\ls s\ls r-2.
\end{align*}
For the cases $r=3,4$ and $s=r-2$, we have, from \eqref{eq.2energy92} and Lemma \ref{lem.CL00lemA.7}, that
\begin{align*}
  &\norm{\Pi\left((\nb^{r-1}u)\cdot \nb^2  p\right)}_{L^2(\D\Omega)}\no\\
  \ls& \norm{\nb^{r-1} u}_{L^2(\D\Omega)}\norm{\nb^2 p}_{L^\infty(\D\Omega)}
  \ls CL\norm{\nb^2 u}_{L^{2(n-1)/(n-2)}(\D\Omega)}\no\\
  \ls &C(K,\vol\Omega)L\left(\norm{\nb^r u}_{L^2(\Omega)}+\norm{\nb^{r-1} u}_{L^2(\Omega)}\right)\no\\
  \ls &C(K,L,\vol\Omega)\left(E_{r-1}^{1/2}(t)+E_r^{1/2}(t)\right).
\end{align*}
For the cases $n=3$, $r=4$ and $s=1$, from \eqref{eq.CL004.48}, Lemma \ref{lem.CL00lemA.7} and \eqref{eq.est.nbrp}, we have
\begin{align*} 
  &\norm{\Pi\left((\nb^2u)\cdot \nb^3  p\right)}_{L^2(\D\Omega)}\no\\
  =&\norm{\Pi\nb^2u\cdot \Pi\nb^3  p+\Pi(\nb^2 u\cdot N)\tilde{\otimes}\Pi(N\cdot \nb^3 p)}_{L^2(\D\Omega)}\no\\
  \ls&C\norm{\Pi\nb^2 u}_{L^4(\D\Omega)}\norm{\Pi\nb^3 p}_{L^4(\D\Omega)}  +C\norm{\Pi(N^a\nb^2 u_a)}_{L^4(\D\Omega)}\norm{\Pi(\nb_N\nb^2  p)}_{L^4(\D\Omega)}\no\\
  \ls &C\norm{\nb^2 u}_{L^4(\D\Omega)}\norm{\nb^3  p}_{L^4(\D\Omega)}\no\\
  \ls&C(K,\vol\Omega)\left(\norm{\nb^3 u}_{L^2(\Omega)}+\norm{\nb^2 u}_{L^2(\Omega)}\right) \left(\norm{\nb^4  p}_{L^2(\Omega)}+\norm{\nb^3  p}_{L^2(\Omega)}\right)\no\\
  \ls&C(K, K_1,\vol\Omega)(E_3^{1/2}(t)+E_2^{1/2}(t))\left(\sum_{s=0}^3 E_s(t)+\left(\sum_{\ell=0}^{2}E_\ell^{1/2}(t)\right)E_4^{1/2}(t)\right)\no\\
  \ls &C(K, K_1,\vol\Omega)\sum_{s=0}^3 E_s(t)\sum_{\ell=0}^4 E_\ell^{1/2}(t).
\end{align*}
Thus, we get
\begin{align*}
  \abs{\eqref{eq.r.e11}}\ls& C(K,K_1,M,M_1,L,1/\eps,\vol\Omega,E_0(0))\Big(1+\sum_{s=0}^{r-1}E_s(t)\Big)\big(1+E_r(t)\big).
\end{align*}
From Lemma \ref{lem.CL00lemA.3}, it follows that
\begin{align*}
  \abs{\eqref{eq.r.e111}+ \eqref{eq.r.e112}}\ls C(K,K_1,M,\vol\Omega,1/\eps)\Big(1+\sum_{s=0}^{r-1} E_s(t)\Big)E_r(t).
\end{align*}
Therefore, we have shown that
\begin{align*}
  &\abs{\eqref{eq.r.e1}+\eqref{eq.r.e2}+\eqref{eq.r.e7}}\\
  \ls& C(K,K_1,M,M_1,L,1/\eps,\vol\Omega,E_0(0)) \Big(1+\sum_{s=0}^{r-1}E_s(t)\Big)\big(1+E_r(t)\big).
\end{align*}

\subsection*{Step 2: Estimate \eqref{eq.r.e3}-\eqref{eq.r.e6} and \eqref{eq.r.e8}} \quad \\
From Lemma \ref{lem.CL00lem2.1}, \eqref{eq.r.u} and \eqref{eq.nb.beta}, we have
\begin{align*}
  &D_t\left(|\nb^{r-1}\curl u|^2+|\nb^{r-1}\curl \F^\top|^2\right)\\
  \detail{=&D_t\left(g^{ac}g^{bd}g^{AF}\nb_A^{r-1}(\curl u)_{ab} \nb_F^{r-1}(\curl u)_{cd}\right)\\
  &+D_t\left(g^{ac}g^{bd}g^{ef}g^{AF}\nb_A^{r-1}(\curl \F^\top)_{abe} \nb_F^{r-1}(\curl \F^\top)_{cdf}\right)\\
  =&(r+1)D_t(g^{ac})g^{bd}g^{AF}\nb_A^{r-1}(\curl u)_{ab} \nb_F^{r-1}(\curl u)_{cd}\\
   &+4 g^{ac}g^{bd}g^{AF}D_t\left(\nb_A^{r-1}\nb_a u_b\right) \nb_F^{r-1}(\curl u)_{cd}\\
   &+(r+1)D_t(g^{ac})g^{bd}g^{ef}g^{AF}\nb_A^{r-1}(\curl \F^\top)_{abe} \nb_F^{r-1}(\curl \F^\top)_{cdf}\\
   &+g^{ac}g^{bd}D_t(g^{ef})g^{AF}\nb_A^{r-1}(\curl \F^\top)_{abe} \nb_F^{r-1}(\curl \F^\top)_{cdf}\\
   &+4 g^{ac}g^{bd}g^{ef}g^{AF}D_t\left(\nb_A^{r-1}\nb_a \F_{be}\right) \nb_F^{r-1}(\curl \F^\top)_{cdf}\\}
  =&-2(r+1)g^{ae}\nb_e u^cg^{bd}g^{AF}\nb_A^{r-1}(\curl u)_{ab} \nb_F^{r-1}(\curl u)_{cd}\\
  &-2(r+1)g^{ae}\nb_e u^cg^{bd}g^{ef}g^{AF}\nb_A^{r-1}(\curl \F^\top)_{abe} \nb_F^{r-1}(\curl \F^\top)_{cdf}\\
  &+2g^{ac}g^{bd}g^{es}\nb_su^f g^{AF}\nb_A^{r-1}(\curl \F^\top)_{abe} \nb_F^{r-1}(\curl \F^\top)_{cdf}\\
  &+ 4g^{ac}g^{bd}g^{AF}\nb_F^{r-1}(\curl u)_{cd}(\curl u)_{be}\nb_{Aa}^r u^e\\
  &+4\sgn(2-r)g^{ac}g^{AF}\nb_F^{r-1}(\curl u)_{cd}\sum_{s=1}^{r-2}\left(r\atop s+1\right)\left((\nb^{1+s}u)\cdot\nb^{r-s}u^d\right)_{Aa}\\
  &+4\sgn(2-r)g^{ac}g^{AF}\nb_F^{r-1}(\curl \F^\top)_{cdf}\sum_{s=1}^{r-2}\left(r\atop s+1\right)\left((\nb^{1+s}u)\cdot\nb^{r-s}\F^{df}\right)_{Aa}\\
&+4 g^{ac}g^{bd}g^{ef}g^{AF}\nb_F^{r-1}(\curl \F^\top)_{cdf}\nb_{Aa}^r\F_{se}\nb_b u^s\\
&- 4g^{ac}g^{bd}g^{ef}g^{AF}\nb_F^{r-1}(\curl \F^\top)_{cdf}\nb_{Aa}^r u^s\nb_s\F_{be}\\
&+4 g^{ac}g^{bd}g^{ef}g^{AF}\nb_F^{r-1}(\curl \F^\top)_{cdf}\nb_eu^s\nb_{Aa}^r\F_{bs}\\
  &+4\nb_f \left(g^{ac}g^{bd}g^{AF}\F^{fe}\nb_F^{r-1}(\curl u)_{cd}\nb_{Aa}^r \F_{be}\right)\\
  &+4g^{ac}g^{bd}g^{AF}\nb_F^{r-1}(\curl u)_{cd}\sum_{s=1}^r\left(r\atop s\right) \left(\nb^s\F^{ef}\nb^{r-s}\nb_e\F^{bf}\right)_{Aa}\\
  &+4g^{ac}g^{AF}\nb_F^{r-1}(\curl \F^\top)_{cdf}\sum_{s=1}^r\left(r\atop s\right) \left(\nb^s\F^{bf}\nb^{r-s}\nb_b u^d\right)_{Aa}.
\end{align*}
Since $N\cdot\F^\top=0$ on $\D\Omega$,  by the H\"older inequality and the Gauss formula, we have
\begin{align}
  \eqref{eq.r.e3}\ls C(K,K_1,M,\vol\Omega,1/\eps)\left(1+\sum_{s=0}^{r-1} E_s(t)\right)E_r(t).
\end{align}
From \eqref{eq.CL00lem3.9.1} and \eqref{eq.Dtpcommu}, we get
  \begin{align*}
    D_t(\nb_N p)\detail{=&D_t(N^a\nb_a p)=(D_t N^a)\nb_a p+N^aD_t\nb_a p\\
    =&(-2h_d^a N^d+h_{NN} N^a)\nb_a p+N^a\nb_a D_tp\\}
    =&-2h_d^a N^d\nb_a p+h_{NN}\nb_N p+\nb_N D_tp,
  \end{align*}
  which implies
  \begin{align}
    \frac{\vartheta_t}{\vartheta}=-\frac{D_t\nb_N  p}{\nb_N p}=\frac{2h_d^a N^d\nb_a  p}{\nb_N p}-h_{NN}+\frac{\nb_ND_t p}{\nb_N p}.
  \end{align}
Hence, \eqref{eq.r.e8} can be controlled by $C(K,M,L,1/\eps)E_r(t)$. The remaining integrals  \eqref{eq.r.e4}, \eqref{eq.r.e5} and \eqref{eq.r.e6} vanish due to the fact $\tr h=0$.

Therefore, we have
\begin{align}
  \frac{d}{dt}E_r(t)\ls &C(K,K_1,M,M_1,L,1/\eps,\vol\Omega,E_0(0))\Big(1+\sum_{s=0}^{r-1}E_s(t)\Big)\big(1+E_r(t)\big),
\end{align}
which implies the desired result \eqref{eq.renergy} by the Gronwall inequality and the induction argument for $r\in\{2,\cdots, n+1\}$.
\end{proof}

\bigskip

\section{Justification of A Priori Assumptions}\label{sec.justifi}

In the derivation of the higher order energy estimates in Section \ref{sec.rorder}, some \textit{a priori} assumptions are made.
In this section we shall justify these  \textit{a priori} assumptions.

Denote
\begin{equation} \label{eq.E}
\begin{split}
  \K(t)=&\max\left(\norm{\theta(t,\cdot)}_{L^\infty(\D\Omega)}, 1/\iota_0(t)\right), \\
  \E(t)=&\norm{1/(\nb_N  p(t,\cdot))}_{L^\infty(\D\Omega)},  \quad \eps(t)=\frac1{\E(t)}. 
\end{split}
\end{equation}

As in Definition \ref{defn.3.5}, 
  let $0<\eps_1<2$ be a fixed number, take $\iota_1=\iota_1(\eps_1)$ to be the largest number such that
 $\abs{\N(\bar{x}_1)-\N(\bar{x}_2)}\ls \eps_1$   whenever  $\abs{\bar{x}_1-\bar{x}_2}\ls \iota_1$ for
 $ \bar{x}_1,\bar{x}_2\in\D\dm$.
 
\begin{lemma}\label{lem.7.6}
  Let $K_1\gs 1/\iota_1$. 
  Then there are continuous functions $G_j$, $j=1,2,3,4$, such that
  \begin{align}
    \norm{\nb u}_{L^\infty(\Omega)}+\norm{\nb \F}_{L^\infty(\Omega)}+\norm{ \F}_{L^\infty(\Omega)}\ls& G_1(K_1,E_0,\cdots, E_{n+1}),\label{eq.e.1}\\
    \norm{\nb  p}_{L^\infty(\Omega)}+\norm{\nb^2 p}_{L^\infty(\D\Omega)}\ls &G_2(K_1,\E,E_0,\cdots, E_{n+1},\vol\Omega),\label{eq.e.2}\\
    \norm{\theta}_{L^\infty(\D\Omega)}\ls &G_3(K_1,\E,E_0,\cdots, E_{n+1},\vol\Omega),\label{eq.e.3}\\
    \norm{\nb D_t p}_{L^\infty(\D\Omega)}\ls &G_4(K_1,\E,E_0,\cdots, E_{n+1},\vol\Omega).\label{eq.e.4}
  \end{align}
\end{lemma}

\begin{proof}
 The estimate \eqref{eq.e.1} follows from \eqref{eq.r.einfu}, \eqref{eq.r.einfbeta} and \eqref{eq.CllemmaA.4}. By Lemmas \ref{lem.CL00lemA.4} and \ref{lem.CL00lemA.2},  we obtain
\begin{align}
  \norm{\nb  p}_{L^\infty(\Omega)}\ls &C(K_1)\sum_{\ell=0}^{2} \norm{\nb^{\ell+1} p}_{L^2(\Omega)},\label{eq.e.5}\\
  \intertext{and}
  \norm{\nb^2 p}_{L^\infty(\D\Omega)}\ls &C(K_1)\sum_{\ell=0}^{n+1} \norm{\nb^{\ell} p}_{L^2(\D\Omega)}.\label{eq.e.6}
\end{align}
Thus, the estimate \eqref{eq.e.2} follows from \eqref{eq.e.5}, \eqref{eq.e.6}, Lemmas \ref{lem.CL00lemA.5}--\ref{lem.CL00lemA.7}, \eqref{eq.deltap2}, \eqref{eq.nb2p} and \eqref{eq.nb3p}. Since $|\nb^2 p|\gs |\Pi\nb^2 p|=|\nb_N p||\theta|\gs \E^{-1}|\theta|$ in view of \eqref{eq.CL004.20}, the estimate \eqref{eq.e.3} follows from \eqref{eq.e.2}. The estimate \eqref{eq.e.4} follows from Lemma \ref{lem.CL00lemA.2}, \eqref{eq.est.nbkdtp}, \eqref{eq.nbdtp} and \eqref{eq.pinbrdtp}.
\end{proof}

\begin{lemma}\label{lem.7.7}
  Let $K_1\gs 1/\iota_1$.
  Then we have
  \begin{align}\label{eq.e.8}
    \abs{\frac{d}{dt}E_r}\ls C_r(K_1,\E,E_0,\cdots, E_{n+1},\vol\Omega)\sum_{s=0}^r E_s,
  \end{align}
  and
  \begin{align}\label{eq.e.9}
    \abs{\frac{d}{dt}\E}\ls C_r(K_1,\E,E_0,\cdots, E_{n+1},\vol\Omega).
  \end{align}
\end{lemma}

\begin{proof}
  The first estimate \eqref{eq.e.8} follows  immediately from Lemma \ref{lem.7.6} and  the proof of Theorems \ref{thm.1energy} and \ref{thm.renergy}. The second estimate follows from
  \begin{align*}
    \abs{\frac{d}{dt}\norm{\frac{1}{-\nb_N p(t,\cdot)}}_{L^\infty(\D\Omega)}}\ls C\norm{\frac{1}{-\nb_N p(t,\cdot)}}_{L^\infty(\D\Omega)}^2 \norm{\nb_ND_t p(t,\cdot)}_{L^\infty(\D\Omega)}
  \end{align*}
  and \eqref{eq.e.4}.
\end{proof}

As a consequence of Lemma \ref{lem.7.7}, we have the following result:

\begin{lemma}\label{lem.7.8}
  There exists a continuous function $\T>0$ depending on $K_1$, $\E(0)$, $E_0(0)$, $\cdots$, $E_{n+1}(0)$ and $\vol\Omega$ such that for
  \begin{align}
    0\ls t\ls \T(K_1,\E(0),E_0(0),\cdots, E_{n+1}(0),\vol\Omega),
  \end{align}
  one has
  \begin{align}\label{eq.7.37}
    E_s(t)\ls 2E_s(0), \quad 0\ls s\ls n+1, \quad \E(t)\ls 2\E(0).
  \end{align}
  Furthermore,  
  \begin{align}\label{eq.7.38}
    \frac{1}{2}g_{ab}(0,y)Y^aY^b\ls g_{ab}(t,y)Y^aY^b\ls 2g_{ab}(0,y)Y^aY^b,
  \end{align}
  and  
  \begin{align}
    \qquad\abs{\N(x(t,\bar{y}))-\N(x(0,\bar{y}))}\ls&\frac{\eps_1}{16}, &&\bar{y}\in\D\Omega,\qquad\label{eq.7.39}\\
    \abs{x(t,y)-x(t,y)}\ls&\frac{\iota_1}{16}, &&y\in\Omega,\label{eq.7.40}\\
    \abs{\frac{\D x(t,\bar{y})}{\D y}-\frac{\D (0,\bar{y})}{\D y}}\ls &\frac{\eps_1}{16}, &&\bar{y}\in\D\Omega.\label{eq.7.41}
  \end{align}
\end{lemma}

\begin{proof}
First when $\T(K_1,\E(0),E_0(0),\cdots, E_{n+1}(0)$, $\vol\Omega)>0$ is sufficiently small,  Lemma \ref{lem.7.7} yields the estimate \eqref{eq.7.37}. Then, from \eqref{eq.7.37} and Lemma \ref{lem.7.6},  
  \begin{align}
    \norm{\nb u}_{L^\infty(\Omega)}+\norm{\nb \F}_{L^\infty(\Omega)} +&\norm{ \F}_{L^\infty(\Omega)}+\norm{\nb  p}_{L^\infty(\Omega)}\no\\
    &\ls C(K_1,\E(0),E_0(0),\cdots, E_{n+1}(0)),\label{eq.7.42}\\
    \norm{\nb^2 p}_{L^\infty(\D\Omega)}+\norm{\theta}_{L^\infty(\D\Omega)} &\ls C(K_1,\E(0),E_0(0),\cdots, E_{n+1}(0),\vol\Omega),\label{eq.7.43}\\
    \intertext{and}
\norm{\nb D_t p}_{L^\infty(\D\Omega)}&\ls C(K_1,\E(0),E_0(0),\cdots, E_{n+1}(0),\vol\Omega).\label{eq.7.44}
  \end{align}
From \eqref{eq.1energy1} and \eqref{eq.1energy2}, we have
  \begin{align*}
    \abs{D_t\nb u}\ls \abs{\nb^2  p}+\abs{\nb u}^2+\abs{\nb \F}^2+\abs{\F}\abs{\nb^2\F},\quad 
    \abs{D_t \nb\F}\ls \abs{\nb\F}\abs{\nb u}+\abs{\F}\abs{\nb^2 u}. 
  \end{align*}
  With the help of \eqref{eq.A.32}, \eqref{eq.CL00lemA.7.1}, Lemma \ref{lem.7.6} and \eqref{eq.7.37}, we obtain
  \begin{align*}
    \norm{\nb u}_{L^\infty(\D\Omega)}+\norm{\nb \F}_{L^\infty(\D\Omega)}\ls C(K_1,\E(0),E_0(0),\cdots, E_{n+1}(0),\vol\Omega).
  \end{align*}
 It follows, from \eqref{eq.7.42}, \eqref{eq.7.43}, Lemmas \ref{lem.CL00lemA.2} and \ref{lem.CL00lemA.7}, \eqref{eq.r.einfu} and \eqref{eq.r.einfbeta}, that
  \begin{align*}
    &\norm{D_t\nb u}_{L^\infty(\D\Omega)}+\norm{D_t\nb\F}_{L^\infty(\D\Omega)}\\
    \ls& \norm{\nb^2 p}_{L^\infty(\D\Omega)}+\left(\norm{\nb u}_{L^\infty(\D\Omega)}+\norm{\nb \F}_{L^\infty(\D\Omega)}\right)^2\no\\
    &+\norm{\F}_{L^\infty(\Omega)}\left(\norm{\nb^2 u}_{L^\infty(\D\Omega)}+\norm{\nb^2 \F}_{L^\infty(\D\Omega)}\right)\no\\
    \ls& C(K_1,\E(0),E_0(0),\cdots, E_{n+1}(0),\vol\Omega)\\
    &\times\left(1+\norm{\nb u}_{L^\infty(\D\Omega)}+\norm{\nb \F}_{L^\infty(\D\Omega)}\right),
  \end{align*}
  which implies, by Gronwall's inequality, for $t\gs 0$
  \begin{align*}
    &\norm{\nb u(t,\cdot)}_{L^\infty(\D\Omega)}+\norm{\nb \F(t,\cdot)}_{L^\infty(\D\Omega)}\no\\
\ls &e^{C(K_1,\E(0),E_0(0),\cdots, E_{n+1}(0),\vol\Omega)t}\left(\norm{\nb u(0,\cdot)}_{L^\infty(\D\Omega)}+\norm{\nb \F(0,\cdot)}_{L^\infty(\D\Omega)}\right)\no\\
    &+e^{C(K_1,\E(0),E_0(0),\cdots, E_{n+1}(0),\vol\Omega)t}-1.
  \end{align*}
   It follows, for $0\ls t\ls T$, that 
  \begin{align}\label{eq.7.45}
    &\norm{\nb u(t,\cdot)}_{L^\infty(\D\Omega)}+\norm{\nb \F(t,\cdot)}_{L^\infty(\D\Omega)}\no\\
\ls& 2\left(\norm{\nb u(0,\cdot)}_{L^\infty(\D\Omega)}+\norm{\nb \F(0,\cdot)}_{L^\infty(\D\Omega)}\right),
  \end{align}
 if we take  $T$ small enough, which also guarantees the \textit{a priori} assumption of \eqref{eq.beta0}.

From \eqref{eq.Dtpcommu}, \eqref{eq.A.4.2}, \eqref{eq.est.nbkdtp}, \eqref{eq.nbdtp} and \eqref{eq.pinbrdtp}, we obtain
 \begin{align*}
   \norm{D_t\nb  p}_{L^\infty(\Omega)}=&\norm{\nb D_t p}_{L^\infty(\Omega)}\ls C(K_1)\sum_{\ell=0}^2\norm{\nb^{\ell+1} D_t p}_{L^2(\Omega)}\\
   \ls & C(K_1,\E(0),E_0(0),\cdots, E_{n+1}(0),\vol\Omega),
 \end{align*}
 which yields,  for  sufficiently small $t>0$, 
 \begin{align}
   \norm{\nb p(t,\cdot)}_{L^\infty(\Omega)}\ls 2\norm{\nb p(0,\cdot)}_{L^\infty(\Omega)}.
 \end{align}
 In view of \eqref{eld41} and \eqref{eq.7.42}, we get
 \begin{align*}
   \norm{D_t v}_{L^\infty(\dm)}\ls& \norm{\D  p}_{L^\infty(\dm)}+\norm{F}_{L^\infty(\dm)}\norm{\D F}_{L^\infty(\dm)}\\
   \ls&\norm{\nb  p}_{L^\infty(\Omega)}+\norm{\F}_{L^\infty(\Omega)} \norm{\nb\F}_{L^\infty(\Omega)}\\
   \ls&C(K_1,\E(0),E_0(0),\cdots, E_{n+1}(0)),
 \end{align*}
 which implies
 \begin{align}\label{eq.7.47}
   \norm{ v(t,\cdot)}_{L^\infty(\dm)}\ls 2\norm{ v(0,\cdot)}_{L^\infty(\Omega)}.
 \end{align}

The relation \eqref{eq.7.38} follows from the same argument because $D_t g_{ab}=\nb_a u_b+\nb_b u_a$ and by \eqref{eq.7.42}
 \begin{align*}
   &\abs{g_{ab}(T,y)Y^aY^b-g_{ab}(0,y)Y^aY^b}   \\
   \ls &\int_0^T\abs{D_t g_{ab}(s,y)} ds Y^aY^b\ls2\int_0^T\norm{\nb_a u_b(s)}_{L^\infty(\Omega)} ds Y^aY^b \ls \frac{1}{2}g_{ab}(0,y)Y^aY^b,
 \end{align*}
 as long as $T$ is sufficiently small. 
 
Recalling the fact $$D_t n_a=h_{NN} n_a,$$
and
 \begin{align*}
   D_t x(t,y)=v(t,x(t,y)),\quad
   D_t\frac{\D x}{\D y} =&\frac{\D  v(t,x(t,y))}{\D y} =\frac{\D  v(t,x)}{\D x}\frac{\D x}{\D y},
 \end{align*}
 one obtains \eqref{eq.7.39}-\eqref{eq.7.41}  from \eqref{eq.7.47} and \eqref{eq.7.45}.
 \end{proof}


As a consequence of \eqref{eq.7.39},  \eqref{eq.7.40} and the triangle inequality, we have the following result:\begin{lemma}\label{lem.7.9}
  Let $\T$ be as in Lemma \ref{lem.7.8}. There exists some $\iota_1>0$ such that, 
  if  $$\abs{\N(x(0,y_1))-\N(x(0,y_2))}\ls \frac{\eps_1}{2}$$ 
  when $\abs{x(0,y_1)-x(0,y_2)}\ls 2\iota_1$,  then 
   $$\abs{\N(x(t,y_1))-\N(x(t,y_2))}\ls \eps_1$$
   when $\abs{x(t,y_1)-x(t,y_2)}\ls 2\iota_1$.

\end{lemma}



Lemmas \ref{lem.7.8} and \ref{lem.7.9} yield immediately the main Theorem \ref{MAIN}.

\bigskip

\appendix

\section{Preliminaries and Some Estimates}

For the convenience of the readers and completeness of preliminary results, we record some definitions and estimates directly from Christodoulou-Lindblad \cite{CL00} in this appendix. 

Let $N^a$ denote the unit normal to $\D\Omega$, $g_{ab}N^aN^b=1$, $g_{ab}N^a T^b=0$ if $T\in T(\D\Omega)$, and let $N_a=g_{ab}N^b$ denote the unit conormal, $g^{ab} N_aN_b=1$. The induced metric $\gamma$ on the tangent space to the boundary $T(\D\Omega)$ extended to be $0$ on the orthogonal complement in $T(\Omega)$ is then given by
\begin{align}
  \gamma_{ab}=g_{ab}-N_aN_b,\quad \gamma^{ab}=g^{ab}-N^aN^b.
\end{align}
The orthogonal projection of an $(r,s)$ tensor $S$ to the boundary is given by
\begin{align}
  (\Pi S)_{b_1\cdots b_s}^{a_1\cdots a_r}=\gamma_{c_1}^{a_1}\cdots \gamma_{c_r}^{a_r}\gamma_{b_1}^{d_1}\cdots \gamma_{b_s}^{d_s} S_{d_1\cdots d_s}^{c_1\cdots c_r},
\end{align}
where
\begin{align}\label{gammaauc}
  \gamma_a^c=\delta_a^c-N_aN^c.
\end{align}
Covariant differentiation on the boundary $\bnb$ is given by
\begin{align}
  \bnb S=\Pi\nb S.
\end{align}
The second fundamental form of the boundary is given by
\begin{align}\label{2ndfundform}
  \theta_{ab}=(\Pi\nb N)_{ab}=\gamma_a^c \nb_c N_b.
\end{align}

Let us now recall some properties of the projection. Since $g^{ab}=\gamma^{ab}+N^aN^b$, we have
\begin{align}\label{eq.CL004.48}
  \Pi(S\cdot R)=\Pi(S)\cdot \Pi(R)+\Pi(S\cdot N)\tilde{\otimes}\Pi(N\cdot R),
\end{align}
where $S\tilde{\otimes} R$ denotes some partial symmetrization of the tensor product $S\otimes R$, i.e., a sum over some subset of the permutations of the indices divided by the number of permutations in that subset. Similarly, we let $S\tilde{\cdot} R$ denote a partial symmetrization of the dot product $S\cdot R$. Now we recall some identities:
\begin{align}
  \Pi\nb^2 q=&\bnb^2 q+\theta \nb_N q,\label{eq.CL004.20}\\
  \Pi\nb^3 q=&\bnb^3 q-2\theta\tilde{\otimes}(\theta\tilde{\cdot}\bnb q)+(\bnb\theta)\nb_N q+3\theta\tilde{\otimes}\bnb\nb_N q.\label{eq.CL004.21} 
\end{align}

\begin{definition}\label{defn.3.3}
  Let $\N(\bar{x})$ be the outward unit normal to $\D\dm$ at $\bar{x}\in \D\dm$. Let $\dist(x_1,x_2)=|x_1-x_2|$ denote the Euclidean distance in $\R^n$, and for $\bar{x}_1, \bar{x}_2\in \D\dm$, let $\dist_{\D\dm} (\bar{x}_1, \bar{x}_2)$ denote the geodesic distance on the boundary.
\end{definition}

\begin{definition}\label{defn.3.4}
  Let $\dist(x,\D\dm)$ be the Euclidean distance from $x$ to the boundary. Let $\iota_0$ be the injectivity radius of the normal exponential map of $\D\dm$, i.e., the largest number such that the map
\begin{align*}
    \D\dm\times (-\iota_0,\iota_0)&\to\{x\in\R^n: \dist(x,\D\dm)<\iota\}\\
    \text{given by } (\bar{x},\iota)&\to x=\bar{x}+\iota \N(\bar{x})
\end{align*}
is an injection.
\end{definition}

\begin{definition}\label{defn.3.5}
  Let $0<\eps_1<2$ be a fixed number, and let $\iota_1=\iota_1(\eps_1)$ the largest number such that
  \begin{align*}
    \abs{\N(\bar{x}_1)-\N(\bar{x}_2)}\ls \eps_1 \quad \text{whenever } \abs{\bar{x}_1-\bar{x}_2}\ls \iota_1, \; \bar{x}_1,\bar{x}_2\in\D\dm.
  \end{align*}
\end{definition}

\begin{lemma}[\mbox{\cite[Lemma 3.9]{CL00}}] \label{lem.CL00lem3.9}
  Let $N$ be the unit normal to $\D\Omega$, and let $h_{ab}=\frac{1}{2}D_tg_{ab}$. On $[0,T]\times \D\Omega$, we have
  \begin{align}\label{eq.CL00lem3.9.1}
    &D_tN_a=h_{NN}N_a,\quad D_tN^c=-2h_d^cN^d+h_{NN}N^c,\;\text{and }
    D_t\gamma^{ab}=-2\gamma^{ac}h_{cd}\gamma^{db},
  \end{align}
  where $h_{NN}=h_{ab}N^aN^b$. The volume element on $\D\Omega$ satisfies
  \begin{align}\label{eq.CL00lem3.9.3}
    D_td\mu_\gamma=(\tr h-h_{NN})d\mu_\gamma=(\tr \theta u\cdot N+\gamma^{ab}\bnb_a\bar{u}_b)d\mu_\gamma,
  \end{align}
  where $\bar{u}_b$ denotes the tangential component of $u_b$ to the boundary $\D\Omega$.
\end{lemma}

\begin{lemma}[\mbox{cf. \cite[Lemma 5.5]{CL00}}] \label{lem.CL00lem5.5}
  Let $w_a=w_{Aa}=\nb_A^r f_a$, $\nb_A^r=\nb_{a_1}\cdots \nb_{a_r}$, $f$ be a $(0,1)$ tensor, and $[\nb_a,\nb_b]=0$. Let $\dv w=\nb_a w^a=\nb^r\dv f$, and let $(\curl w)_{ab}=\nb_aw_b-\nb_b w_a=\nb^r(\curl f)_{ab}$. Then,
  \begin{align}
    |\nb w|^2\ls C(g^{ab}\gamma^{cd}\gamma^{AB}\nb_c w_{Aa}\nb_d w_{Bb}+|\dv w|^2+|\curl w|^2).
  \end{align}
\end{lemma}

\begin{lemma}[\mbox{\cite[Proposition 5.8]{CL00}}] \label{lem.CL00prop5.8}
Let $\iota_0$ and $\iota_1$ be as in Definitions \ref{defn.3.4} and \ref{defn.3.5}, and suppose that $|\theta|+1/\iota_0\ls K$ and $1/\iota_1\ls K_1$. Then with $\tilde{K}=\min(K,K_1)$ we have, for any $r\gs 2$ and $\delta>0$,
\begin{align}
  &\norm{\nb^r q}_{L^2(\D\Omega)}+\norm{\nb^r q}_{L^2(\Omega)}\no\\
  &\qquad\ls C\norm{\Pi \nb^r q}_{L^2(\D\Omega)}+C(\tilde{K},\vol\Omega)\sum_{s\ls r-1} \norm{\nb^s\Delta q}_{L^2(\Omega)},\no\\
  &\norm{\nb^{r-1} q}_{L^2(\D\Omega)}+\norm{\nb^r q}_{L^2(\Omega)}\no\\
  &\qquad\ls \delta\norm{\Pi \nb^r q}_{L^2(\D\Omega)}+C(1/\delta,K,\vol\Omega)\sum_{s\ls r-2} \norm{\nb^s\Delta q}_{L^2(\Omega)}.\label{eq.CL00prop5.8.2}
\end{align}
\end{lemma}

\begin{lemma}[cf. \mbox{\cite[Proposition 5.9]{CL00}}] \label{lem.CL00prop5.9}
  Assume that $0\ls r\ls 4$. Suppose that $|\theta|\ls K$ and $\iota_1\gs 1/K_1$, where $\iota_1$ is as in Definition 3.5 of \cite{CL00}. If $q=0$ on $\D\Omega$, then for $m=0,1$,
  \begin{align}\label{eq.CL00prop5.9.1}
    \norm{\Pi\nb^r q}_{L^2(\D\Omega)}\ls & C(K,K_1)\left(\norm{\theta}_{L^\infty(\D\Omega)}+\sum_{k\ls r-2-m} \norm{\bnb^k \theta}_{L^2(\D\Omega)}\right)\no \\
    &\times\sum_{k\ls r-2+m}\norm{\nb^k q}_{L^2(\D\Omega)}.
  \end{align}
  If, in addition, $|\nb_N q|\gs \eps>0$ and $|\nb_N q|\gs 2\eps \norm{\nb_N q}_{L^\infty(\D\Omega)}$, then
  \begin{align*} 
    &\norm{\bnb^{r-2}\theta}_{L^2(\D\Omega)} \ls C\left(K,K_1,\frac{1}{\eps}\right) \Big(\norm{\theta}_{L^\infty(\D\Omega)}+\sum_{k\ls r-3} \norm{\bnb^k \theta}_{L^2(\D\Omega)}\Big)\sum_{k\ls r-1} \norm{\nb^k q}_{L^2(\D\Omega)}.\no
  \end{align*}
\end{lemma}

\begin{lemma}[cf. \mbox{\cite[Proposition 5.10]{CL00}}] \label{lem.CL00prop5.10}
  Assume that $0\ls r\ls 4$ and that $|\theta|+1/\iota_0\ls K$. If $q=0$ on $\D\Omega$, then
  \begin{align*}
    \norm{\nb^{r-1} q}_{L^2(\D\Omega)} \ls& C\left(\norm{\bnb^{r-3}\theta}_{L^2(\D\Omega)} \norm{\nb_N q}_{L^\infty(\D\Omega)}+\norm{\nb^{r-2}\Delta q}_{L^2(\Omega)}\right)\no\\
    &+C\left(K,\vol\Omega,\norm{\theta}_{L^2(\D\Omega)}\right)\left(\norm{\nb_N q}_{L^\infty(\D\Omega)}+\sum_{s=0}^{r-3} \norm{\nb^s\Delta q}_{L^2(\Omega)}\right).
  \end{align*}
\end{lemma}

\begin{lemma}[\mbox{\cite[Lemma A.1]{CL00}}] \label{lem.CL00lemA.1}
   Let $2\ls p\ls s\ls q\ls \infty $ and $\frac{m}{s}=\frac{k}{p}+\frac{m-k}{q}$. If $\alpha$ is a $(0,r)$ tensor, then with $a=k/m$ and a constant $C$ that only depends on $m$ and $n$, such that
  \begin{align*}
    \norm{\bnb^k\alpha}_{L^s(\D\Omega)}\ls C\norm{\alpha}_{L^q(\D\Omega)}^{1-a}\norm{\bnb^m \alpha}_{L^p(\D\Omega)}^a.
  \end{align*}
\end{lemma}

\begin{lemma}[\mbox{\cite[Lemma A.2]{CL00}}] \label{lem.CL00lemA.2}
  Suppose that for $\iota_1\gs 1/K_1$
  \begin{align*}
    \abs{\N(\bar{x}_1)-\N(\bar{x}_2)}\ls \eps_1, \quad \text{whenever } |\bar{x}_1-\bar{x}_2|\ls \iota_1, \; \bar{x}_1,\bar{x}_2\in\D\dm,
  \end{align*}
  and
  \begin{align*}
    C_0^{-1}\gamma_{ab}^0(y) Z^aZ^b\ls \gamma_{ab}(t,y)Z^aZ^b\ls C_0\gamma_{ab}^0(y) Z^aZ^b, \quad \text{if } Z\in T(\Omega),
  \end{align*}
  where $\gamma_{ab}^0(y)=\gamma_{ab}(0,y)$. Then if $\alpha$ is a $(0,r)$ tensor,
  \begin{align}
    &\norm{\alpha}_{L^{(n-1)p/(n-1-kp)}(\D\Omega)}\ls C(K_1) \sum_{\ell=0}^k \norm{\nb^\ell \alpha}_{L^p(\D\Omega)}, \quad 1\ls p<\frac{n-1}{k},\\
    &\norm{\alpha}_{L^\infty(\D\Omega)}\ls \delta\norm{\nb^k \alpha}_{L^p(\D\Omega)}+C_\delta(K_1)\sum_{\ell=0}^{k-1} \norm{\nb^\ell \alpha}_{L^p(\D\Omega)}, \quad k>\frac{n-1}{p},\label{eq.A.32}
  \end{align}
  for any $\delta>0$.
\end{lemma}

\begin{lemma}[\mbox{\cite[Lemma A.3]{CL00}}] \label{lem.CL00lemA.3}
  With notation as in Lemmas \ref{lem.CL00lemA.1} and \ref{lem.CL00lemA.2}, we have
  \begin{align*}
    \sum_{j=0}^k\norm{\nb^j\alpha}_{L^s(\Omega)}\ls C\norm{\alpha}_{L^q(\Omega)}^{1-a}\left(\sum_{i=0}^m \norm{\nb^i\alpha}_{L^p(\Omega)}K_1^{m-i}\right)^a.
  \end{align*}
\end{lemma}

\begin{lemma}[\mbox{\cite[Lemma A.4]{CL00}}] \label{lem.CL00lemA.4}
  Suppose that $\iota_1\gs 1/K_1$ and $\alpha$ is a $(0,r)$ tensor. Then
  \begin{align}
    \norm{\alpha}_{L^{np/(n-kp)}(\Omega)} \ls& C\sum_{\ell=0}^k K_1^{k-\ell} \norm{\nb^\ell \alpha}_{L^p(\Omega)}, \quad 1\ls p<\frac{n}{k},\label{eq.A.4.1}\\
    \norm{\alpha}_{L^\infty(\Omega)}\ls &C\sum_{\ell=0}^k K_1^{n/p-\ell} \norm{\nb^\ell \alpha}_{L^p(\Omega)}, \quad k>\frac{n}{p}.\label{eq.A.4.2}
  \end{align}
\end{lemma}

\begin{lemma}[\mbox{\cite[Lemma A.5]{CL00}}] \label{lem.CL00lemA.5}
  Suppose that $q=0$ on $\D\Omega$. Then
  \begin{align*}
    \norm{q}_{L^2(\Omega)}\ls C(\vol\Omega)^{1/n}\norm{\nb q}_{L^2(\Omega)},\quad
    \norm{\nb q}_{L^2(\Omega)}\ls C(\vol\Omega)^{1/2n}\norm{\Delta q}_{L^2(\Omega)}.
  \end{align*}
\end{lemma}

\begin{lemma}[\mbox{\cite[Lemma A.7]{CL00}}] \label{lem.CL00lemA.7}
  Let $\alpha$ be a $(0,r)$ tensor. Assume that
  $$\vol\Omega \ls V \text{ and  }\norm{\theta}_{L^\infty(\D\Omega)}+1/\iota_0 \ls K,$$
  then there is a $C=C(K,V,r,n)$ such that
  \begin{align}
    &\norm{\alpha}_{L^{(n-1)p/(n-p)}(\D\Omega)} \ls C\norm{\nb \alpha}_{L^p(\Omega)} +C\norm{\alpha}_{L^p(\Omega)},\quad 1\ls p<n,\label{eq.CL00lemA.7.1}\\
    &\norm{\nb^2\alpha}_{L^2(\Omega)} \ls C\left(\norm{\Pi\nb^2\alpha}_{L^{2(n-1)/n}(\D\Omega)} +\norm{\Delta\alpha}_{L^2(\Omega)}+\norm{\nb\alpha}_{L^2(\Omega)}\right). \label{eq.CL00lemA.7.2}
  \end{align}
\end{lemma}

\bigskip

\section*{Acknowledgments}

The authors thank Professor Tao Luo for some helpful discussions. C. Hao was partially supported by National Science
Foundation of China under Grant 11171327 and by the Youth Innovation Promotion Association, Chinese Academy of Sciences.
D. Wang's research was supported in part by the National Science
Foundation under Grant DMS-1312800.


\bigskip


\end{document}